\documentclass[11pt]{article}

\usepackage{amsmath,amsthm,amssymb,xypic}
\usepackage{hyperref}

\theoremstyle{plain}
    \newtheorem{theorem}{Theorem}[section]
    \newtheorem{lemma}[theorem]{Lemma}
    \newtheorem{corollary}[theorem]{Corollary}
    \newtheorem{proposition}[theorem]{Proposition}

 \theoremstyle{definition}
    \newtheorem{definition}[theorem]{Definition}
    
    \newtheorem{example}[theorem]{Example}

    \newtheorem{remark}[theorem]{Remark}

\theoremstyle{remark}

\numberwithin{equation}{section}

 \DeclareMathOperator{\Tr}{Tr}
 \DeclareMathOperator{\tr}{tr}

\DeclareMathOperator{\im}{im}

\DeclareMathOperator{\Ad}{Ad}

\DeclareMathOperator{\End}{End}

\DeclareMathOperator{\rank}{rank}

         \DeclareMathOperator{\supp}{supp}

\DeclareMathOperator{\WF}{WF}
\DeclareMathOperator{\Id}{Id}

\begin{document}

    \newcommand{\R}{\mathbb{R}}
    \newcommand{\C}{\mathbb{C}}
    \newcommand{\N}{\mathbb{N}}
    \newcommand{\Z}{\mathbb{Z}}
    \newcommand{\Q}{\mathbb{Q}}
    \newcommand{\bT}{\mathbb{T}}
    \newcommand{\bP}{\mathbb{P}}

\newcommand{\kg}{\mathfrak{g}}
\newcommand{\ka}{\mathfrak{a}}
\newcommand{\kb}{\mathfrak{b}}
\newcommand{\kk}{\mathfrak{k}}
\newcommand{\kt}{\mathfrak{t}}
\newcommand{\kp}{\mathfrak{p}}
\newcommand{\km}{\mathfrak{m}}
\newcommand{\kh}{\mathfrak{h}}
\newcommand{\kso}{\mathfrak{so}}

\newcommand{\cA}{\mathcal{A}}
\newcommand{\cE}{\mathcal{E}}
\newcommand{\calL}{\mathcal{L}}
\newcommand{\calH}{\mathcal{H}}
\newcommand{\cO}{\mathcal{O}}
\newcommand{\cB}{\mathcal{B}}
\newcommand{\cK}{\mathcal{K}}
\newcommand{\cP}{\mathcal{P}}
\newcommand{\cN}{\mathcal{N}}
\newcommand{\calD}{\mathcal{D}}
\newcommand{\cC}{\mathcal{C}}
\newcommand{\calS}{\mathcal{S}}
\newcommand{\cM}{\mathcal{M}}
\newcommand{\cU}{\mathcal{U}}
\newcommand{\cT}{\mathcal{T}}

\newcommand{\Rnz}{\R \setminus \{0\}}

\newcommand{\Bigwedge}{\textstyle{\bigwedge}}

\newcommand{\beq}[1]{\begin{equation} \label{#1}}
\newcommand{\eeq}{\end{equation}}

\newcommand{\ddt}{\left. \frac{d}{dt}\right|_{t=0}}
\newcommand{\mattwo}[4]{
\left( \begin{array}{cc}
#1 & #2 \\ #3 & #4
\end{array}
\right)
}

\newcommand{\dds}{\left. \frac{d}{ds}\right|_{s=0}}

\title{An equivariant Guillemin trace formula}

\author{Peter Hochs\footnote{Radboud University, \texttt{p.hochs@math.ru.nl}} {}
and Hemanth Saratchandran\footnote{The University of Adelaide, \texttt{hemanth.saratchandran@adelaide.edu.au}}}

%\address{Institute for Mathematics, Astrophysics and Particle Physics, Radboud University}
%\email{p.hochs@math.ru.nl}
%
%\address{School of Mathematical Sciences, The University of Adelaide}
%\email{hemanth.saratchandran@adelaide.edu.au}

%\author{Peter Hochs
%and Hemanth Saratchandran}

%\address{Institute for Mathematics, Astrophysics and Particle Physics, %Radboud University}
%\email{p.hochs@math.ru.nl}

%\address{School of Mathematical Sciences, The University of Adelaide}
%\email{hemanth.saratchandran@adelaide.edu.au}

\date{\today}

\maketitle

\begin{abstract}
Guillemin's trace formula is an expression for the distributional trace of an operator defined by pulling back functions  along a flow on a compact manifold. We obtain an equivariant generalisation of this formula, for proper, cocompact group actions. This is motivated by the construction of an equivariant version of the Ruelle dynamical $\zeta$-function in another paper by the same authors, which is based on the equivariant Guillemin trace formula. To obtain this formula, we first develop an equivariant version of the distributional trace that appears in Guillemin's formula and other places.
\end{abstract}

\tableofcontents

\section{Introduction}

\subsection*{Guillemin's trace formula}

Consider a flow $\varphi$ on a compact manifold $M$. For a flow curve $\gamma$ of $\varphi$ that is periodic with period $l$, let the linearised Poincar\'e map
\[
P_{\gamma} \colon T_{\gamma(0)}M/\R \gamma'(0) \to T_{\gamma(0)}M/\R \gamma'(0)
\]
be induced by the derivative of the flow map $\varphi_l$. We assume that $\varphi$ is nondegenerate, in the sense that $1-P_{\gamma}$ is invertible for all such $\gamma$.

For $t \in \R$, consider the operator $\varphi_t^*$ on $C^{\infty}(M)$ given by pulling back functions along $\varphi_t$. If $\psi \in C^{\infty}_c(\Rnz)$, then we define the operator $\varphi^*_{\psi}$ on $C^{\infty}(M)$ by
\[
\varphi^*_{\psi} := \int_{\Rnz} \psi(t) \varphi_t^*\, dt.
\]
Nondegeneracy of $\varphi$ implies that the wave front set of the Schwartz kernel $K_{\varphi^*_{\psi}}$ of $\varphi^*_{\psi}$ is disjoint from the conormal bundle to the diagonal  $\Delta(M) \subset M \times M$. Hence $K_{\varphi^*_{\psi}}$ has a well-defined restriction to $\Delta(M)$. Then the distributional \emph{flat trace} of $\varphi_{\psi}^*$ is
\beq{eq def flat trace intro}
\Tr^{\flat}(\varphi^*_{\psi}) :=
\int_{\Delta(M)} K_{\varphi^*_{\psi}}|_{\Delta(M)}\, dm,
\eeq
where this integral is the pairing of $K_{\varphi^*_{\psi}}|_{\Delta(M)} \in \calD'(\Delta(M))$ with the constant function $1$. This trace is used for example in Theorem 19.4.1 in \cite{Hormander3}.

Guillemin's trace formula (Theorem 8 in \cite{Guillemin77}) states that $\psi \mapsto \Tr^{\flat}(\varphi^*_{\psi}) $ defines a distribution $\Tr^{\flat}(\varphi^*) $ on $\Rnz$, given by
\beq{eq GTF intro}
\Tr^{\flat}(\varphi^*) = \sum_{\gamma(l) = \gamma(0)} \frac{T^{\#}_{\gamma} \delta_l}{|\det(1-P_{\gamma})|}.
\eeq
Here the sum is over all periodic flow curves (modulo time shifts), with all possible periods $l$, and $T^{\#}_{\gamma} $ is the primitive period of such a curve.

One application of Guillemin's trace formula \eqref{eq GTF intro} is that it is the basis of proofs of analytic properties of the \emph{Ruelle dynamical $\zeta$-function} \cite{Shen21} for Anosov flows, such as meromorphic continuation (see Section 2.2 in \cite{GLP13}, or \cite{DZ16}) and homotopy invariance (see Theorem 2 in \cite{DGRS20}).

\subsection*{An equivariant Guillemin trace formula}

Our goal in this paper is to generalise Guillemin's trace formula \eqref{eq GTF intro} to an equivariant version, for proper, cocompact group actions by unimodular, locally compact groups. In another paper \cite{HS25b}, we use this generalisation to construct an eqiuvariant version of the Ruelle dynamical $\zeta$-function. We then compare this with equivariant analytic torsion \cite{HS22a} and investigate an equivariant generalisation of the Fried conjecture \cite{Fried87, Shen21}. In \cite{HS25b}, direct consequences of the equivariant Guillemin trace formula are that the definition of the equivariant Ruelle dynamical $\zeta$-function is independent of choices made, and a relation with the classical Ruelle dynamical $\zeta$-function. Analogously to the non-equivariant case, it is possible that the equivariant Guillemin trace formula is the basis for future results on analytic properties of the equivariant Ruelle dynamical $\zeta$-function.

We no longer assume that $M$ is compact, but suppose that a unimodular, locally compact group $G$ acts properly on $M$, such that $M/G$ is compact. Then there is a function $\chi \in C^{\infty}_c(M)$ such that for all $m \in M$, with respect to a Haar measure $dx$, 
\[
\int_G \chi(xm)\, dx = 1.
\]
Let $g \in G$, and suppose that its centraliser $Z$ is unimodular. Then there is a unique $G$-invariant measure $d(hZ)$ on $G/Z$ compatible with given Haar measures on $G$ and $Z$. In several places, for example \cite{HWW, HW2, PPST21}, the \emph{$g$-trace} of operators with smooth kernels is used to obtain equivariant index theorems in different settings. A version  for discrete groups was used before to construct higher invariants, see for example \cite{Lott99, Wangwang, XieYu}. 
Let $E \to M$ be a $G$-equivariant, Hermitian vector bundle. Let $T$ be an operator from $\Gamma^{\infty}_c(E)$ to $\Gamma^{\infty}(E)$ with a smooth Schwartz kernel $\kappa_T$. If the following converges, then the $g$-trace of $T$ is
\beq{eq g trace intro}
\Tr_g(T) := \int_{G/Z} \int_M\chi(m) \tr(hgh^{-1}\kappa_T(hg^{-1}h^{-1}m,m))\, dm\, d(hZ),
\eeq
where $dm$ is a given $G$-invariant density on $M$.
 
As a common generalisation of the traces  \eqref{eq def flat trace intro} and \eqref{eq g trace intro}, we start by developing the \emph{flat $g$-trace} $\Tr^{\flat}_g$ in Definition \ref{def flat g trace}. This a generalisation of \eqref{eq g trace intro} to operators with distributional Schwartz kernels with suitable transversality properties.

Using this flat $g$-trace, we generalise Guillemin's trace formula \eqref{eq GTF intro} to $G$-equivariant flows satisfying a suitable nondegeneracy condition. This is our main result, Theorem \ref{prop fixed pt gen}. It also involves a lift of $\varphi$ to a $G$-equivariant vector bundle, and composition with a $G$-equivariant endomorphism of this bundle. The basic idea is that the sum over periodic flow curves in \eqref{eq GTF intro} is replaced by flow curves $\gamma$ such that $\gamma(l) = hgh^{-1}\gamma(0)$ for $l \in \Rnz$ and $h \in G$.

\subsection*{Acknowledgements}

We thank Jean-Michel Bismut, Bingxiao Liu, Thomas Schick, Shu Shen and Polyxeni Spilioti for helpful comments and discussions. PH is partially supported by the Australian Research Council, through Discovery Project DP200100729, and by NWO ENW-M grant OCENW.M.21.176.
HS was supported by the Australian Research Council, through grant FL170100020.

\section{Preliminaries and results}

\subsection{The flat $g$-trace}\label{sec flat g trace}

The main result in this paper is
the equivariant Guillemin trace formula, Theorem \ref{prop fixed pt gen}. It  is stated in terms of 
the \emph{flat $g$-trace} that we introduce in this subsection.
%
%It also suggests that Definition \ref{def equivar Ruelle} is a natural generalisation of the classical Ruelle dynamical $\zeta$-function. Finally, this expression could be useful in the future, to prove other properties of the  equivariant Ruelle dynamical $\zeta$-function. 
%
This flat $g$-trace is a common generalisation of the distributional trace used for example in  Theorem 19.4.1 in \cite{Hormander3}, and a $g$-trace defined in terms of orbital integrals, and used for example in \cite{HWW, HW2, PPST21}.

Let  $M$ be a  manifold, on which a unimodular, locally compact group $G$ acts properly  and cocompactly (by which we mean that $M/G$ is compact). 
Fix a $G$-equivariant vector bundle $E \to M$. Fix a Haar measure $dx$ on $G$. 
Let $g \in G$, and let $Z<G$ be ts centraliser. Suppose that $Z$ is unimodular; then there is a $G$-invariant measure $d(hZ)$ on $G/Z$.

For notational simplicity, we fix a $G$-invariant, nonvanishing density $dm$ on $M$, and use this to identify the space $ \Gamma^{-\infty}(E)$ of generalised sections of $E$ with the continuous dual of $\Gamma_c^{\infty} (E^*)$. What follows does not depend on the choice of this density, see for example Remark \ref{rem flat g trace indep metric}. 
We use analogous identifications for other vector bundles as well. For example, we identify 
\[
 \Gamma^{-\infty}( (E \boxtimes E^*)|_{\Delta(M)})\cong \Gamma_c^{\infty} (E^* \boxtimes E)|_{\Delta(M)})'
\]
where the prime on the right hand side denotes the continuous dual space, and $\Delta(M) \subset M \times M$ is the diagonal.

%\subsection{Distributional kernels on the diagonal}

We define a flat $g$-trace of  operators on $\Gamma^{\infty}(E)$ in Definition \ref{def flat g trace}. We later apply this to flows in  \eqref{eq flat g trace flow}. 

%Let $\Delta(M) \subset M \times M$ be the diagonal.
%Using the Riemannian density on $M$ to omit density bundles for notational simplicity, we identify the space $\Gamma^{-\infty}( (E \boxtimes E^*)|_{\Delta(M)})$ of generalised sections with the continuous dual of $\Gamma_c^{\infty} (E^* \boxtimes E)|_{\Delta(M)})$.

The identity operator $\Id_E$ on $E$ defines an element of
\beq{eq isom IdE}
 \Gamma^{\infty}(E \boxtimes E^*|_{\Delta(M)})\cong\Gamma^{\infty}(E^* \boxtimes E|_{\Delta(M)}),
 \eeq
  where the isomorphism is swapping the components in $E$ and $E^*$.
\begin{definition}\label{def fibre trace}
The \emph{fibre-wise trace} of $\kappa \in \Gamma^{-\infty}( (E \boxtimes E^*)|_{\Delta(M)})$ is the distribution $\tr(\kappa) \in \calD'(\Delta(M))$ defined by
\[
\langle \tr(\kappa),f\rangle := \langle \kappa, f \Id_E\rangle,
\]
for $f \in C^{\infty}_c(\Delta(M))$.

\end{definition}

\begin{example}\label{ex fibre trace smooth}
In the case where $\kappa \in \Gamma^{\infty}( (E \boxtimes E^*)|_{\Delta(M)})$ is smooth, let $m \in M$, let $\{e_j\}_{j=1}^{\rank(E)}$ be a basis of $E_m$, and $\{e^j\}_{j=1}^{\rank(E)}$ be the dual basis of $E^*_m$. Then under the isomorphism \eqref{eq isom IdE},
\[
\Id_{E_m}\mapsto \sum_{j} e^j \otimes e_j. 
\] 
Therefore, viewing $\kappa(m,m)$ as an endomorphism of $E_m$, we have
\[
\langle \kappa(m,m), \Id_{E_m}\rangle = \sum_j \langle e^j, \kappa(m,m)e_j \rangle =\tr(\kappa(m,m)),
\]
where $\tr$ is the fibre-wise operator trace.
It follows that $\tr(\kappa)$ is indeed the fibre-wise trace of $\kappa$, and lies in $C^{\infty}(\Delta(M))$.
\end{example}

Let $T\colon \Gamma^{\infty}_c(E) \to \Gamma^{\infty}(E)$ be a $G$-equivariant, continuous linear operator. Let $\kappa$ be its Schwartz kernel. For $x \in G$, we write $x \cdot \kappa$ for the Schwartz kernel of $x \circ T$, where $x$ acts on sections in the usual way. 
Let $\cN^*(\Delta(M)) \to \Delta(M)$ be the conormal bundle of $\Delta(M)$ in $M \times M$.

%Then $hgh^{-1} \cdot \kappa$ has a well-defined restriction to $\Delta(M)$ for all $h \in G$. Applying Definition \ref{def fibre trace}, we obtain $\tr(hgh^{-1} \cdot \kappa|_{\Delta(M)}) \in \calD'(M)$. 

Because the action by $G$ on $M$ is proper and cocompact, there is a \emph{cutoff function} $\chi \in C^{\infty}_c(M)$ such that for all $m \in M$
\beq{eq cutoff fn}
\int_G \chi(xm)\, dx = 1.
\eeq
Let $p_2 \colon M \times M \to M$ be projection onto the second coordinate. Then $(p_2^*\chi)|_{\Delta(M)}$ has compact support. Note that the function $(p_2^*\chi)|_{\Delta(M)}$ is just the function on $\Delta(M)$ corresponding to $\chi$ via the diagonal diffeomorphism $\Delta\colon M \to \Delta(M)$. The wave front set of a distribution is denoted by $\WF$; see Definition 8.1.2 in \cite{Hormander1}. 
\begin{definition}\label{def flat g trace}
Suppose that for all $h \in G$, 
\beq{eq condition WF K}
\WF(hgh^{-1} \cdot \kappa) \cap \cN^*(\Delta(M)) = \emptyset.
\eeq
Suppose also that for all $f \in C^{\infty}_c(M)$,  the integral
\beq{eq def flat g trace 1}
\int_{G/Z} \langle \tr( (hgh^{-1} \cdot \kappa)|_{\Delta(M)}),  (p_2^*f)|_{\Delta(M)}\rangle \, d(hZ)
\eeq
converges absolutely, and that this defines a distribution on $M$. 
Then we say that $T$ is \emph{flat $g$-trace class}, and define its  \emph{flat $g$-trace}  by
\beq{eq def flat g trace}
\Tr_g^{\flat}(T) := \int_{G/Z} \langle \tr( (hgh^{-1} \cdot \kappa)|_{\Delta(M)}),  (p_2^*\chi)|_{\Delta(M)}\rangle \, d(hZ).
\eeq
\end{definition} 
\begin{remark}
The condition \eqref{eq condition WF K} implies that the restriction $(hgh^{-1} \cdot \kappa)|_{\Delta(M)}$ in \eqref{eq def flat g trace 1} and \eqref{eq def flat g trace} is well-defined. See e.g.\ Theorem 8.2.4 in \cite{Hormander1}.
\end{remark}

\begin{remark} \label{rem flat g trace indep metric}
If we view $\kappa$ as a generalised section in $\Gamma^{-\infty}(E \boxtimes E^*)$, this implicitly involves the density $dm \times dm$. The identification of $\Gamma^{-\infty}(E \boxtimes E^*)$ with $\Gamma_c^{\infty}(E^* \boxtimes E)'$ via the same density, used in \eqref{eq def flat g trace}, cancels out the choice of density in a sense, so that  $\Tr_g^{\flat}(T) $ does not depend on $dm$. Compare Lemma 2.9 in \cite{HS22a} in the case where $\kappa$ is smooth. 
\end{remark}

\begin{example}\label{ex class flat trace}
If $G$ is the trivial group, so $g=e$, then $\chi$ may be taken to be the constant $1$. So the flat $e$-trace is the classical flat trace, used for example in Theorem 19.4.1 in \cite{Hormander3}.
\end{example}
\begin{example}\label{eq Trg smooth}
If $\kappa$ is smooth, then by Example \ref{ex fibre trace smooth}, 
\[
\Tr_g^{\flat}(T) = \int_{G/Z} \int_M \chi(m) \tr(hgh^{-1}\kappa(hg^{-1}h^{-1}m,m))\, dm\, d(hZ).
\]
If $T\chi$ is a trace-class operator on $L^2(E)$ and the integral converges, then this equals the 
\emph{$g$-trace} 
\beq{eq g trace}
\Tr_g(T) := \int_{G/Z} \Tr(hgh^{-1} T\chi)\, d(hZ).
\eeq
This trace has been used in various places, see for example \cite{HS22a, HWW, HW2, Lott99, Wangwang}.
In this sense, the flat $g$-trace $\Tr_g^{\flat}$ is a generalisation of the $g$-trace $\Tr_g$, in the same way that the classical flat trace, see Example \ref{ex class flat trace}, generalises the operator trace in some settings.
\end{example}
\begin{example}
If $g=e$, then $\Tr_e^{\flat} =  \langle \tr(\kappa|_{\Delta(M)}),  (p_2^*\chi)|_{\Delta(M)}\rangle$,  a distributional version of the von Neumann trace \cite{Atiyah76, Wang14}.
\end{example}

\begin{proposition}\label{prop flat g trace indep chi}
The integral \eqref{eq def flat g trace} is independent of the function $\chi \in C^{\infty}_c(M)$ satisfying \eqref{eq cutoff fn}.
\end{proposition}
This proposition is proved in Subsection \ref{sec indep chi}.

\subsection{Equivariant flows} \label{sec equivar flows}

Let $\varphi$ be the flow of a smooth vector field $u$ on $M$. Suppose that $u$ has no zeroes, and that its flow $\varphi$ is defined for all time. 
Let $G$ be a unimodular locally compact group, acting properly and cocompactly on  $M$.
%Suppose that the tangent maps of the maps $\varphi_t$ preserve the Riemannian metric.  
Suppose that for all $t \in \R$, the flow map $\varphi_t\colon M \to M$ is $G$-equivariant.
This is true in the following special cases.
\begin{lemma}\label{lem geod flow invar}
Suppose that $G$ acts isometrically on a Riemannian manifold $X$, that $M = S(TX)$ is the unit sphere bundle of $TX$, and that $\varphi$ is the geodesic flow. Then $\varphi$ commutes with the action by $G$. 
\end{lemma}
\begin{proof}
As $G$ acts by isometries, its elements map geodesics to geodesics. This implies the claim.
%
%For $x \in X$ and $v \in T_xX$ of unit length, and for $g \in G$ and $t \in \R$,
%\[
%\begin{split}
%g\cdot 
%\varphi_t(v) &= T_{\exp_x(tv)} g \left. \frac{d}{ds}\right|_{s=t} \exp_x(sv) \quad \in T_{g \exp_x(tv)}X\\
%&=  \left. \frac{d}{ds}\right|_{s=t} g\exp_x(sv)\\
%&=  \left. \frac{d}{ds}\right|_{s=t} \exp_{gx}(sT_x g (v))\\
%&= \varphi_t(g \cdot v). 
%\end{split}
%\]
%In the third equality, we used the fact that $g$ acts on $X$ by isometries.
\end{proof}

\begin{example} \label{ex lift univ cover}
Let $X$ be a compact manifold. Let $M$ be the universal cover of $X$, acted on by the fundamental group $G$ of $X$. A flow $\varphi_X$ on $X$ has a unique $G$-equivariant lift to a flow $\varphi$ on $M$. The generating vector field $u$ of $\varphi$ has no zeroes if and only if the generating vector field of $\varphi_X$ has no zeroes.
\end{example}

\begin{example}
Suppose that $G$ is a compact Lie group, acting on a compact manifold $X$.
Let $S(\kg)$ be the unit sphere in the Lie algebra $\kg$ of $G$, with respect to an $\Ad(G)$-invariant inner product. Let $M = X \times S(\kg)$. Define the action by  $G$ on $M$, and the flow $\varphi$,  by
\[
\begin{split}
g\cdot (p, Y) &= (gp, \Ad(g)Y);\\
\varphi_t(p,Y) &= (\exp(tY)p, Y),
\end{split}
\] 
for $g \in G$, $p \in X$, $Y \in S(\kg)$ and $t \in \R$. Then $\varphi$ is equivariant with respect to the action by $G$. The generating vector field $u$ is given by $u(p,Y) = (Y^X_p, 0)$, where $Y^X_p$ is the tangent vector at $p \in X$ induced by $Y \in S(\kg)$. It follows that $u$ has no zeroes if and only if the action is locally free.
%
%\Todo: is this even g-Anosov? Seems we need $g \not \in \exp(\kg)$. Only leave in this example if there are examples where all conditions hold.
\end{example}

%\subsection{The case of equivariant flows} \label{sec flat trace flows}
%
%As we did in Subsection \ref{sec equivar flows}, we consider a $G$-equivariant flow  $\varphi$ on $M$, generated by a nonvanishing vector field $u$. 

Now we assume that $\Phi$ is a $G$-equivariant, fibre-wise linear flow on $E$ that lifts $\varphi$.
% (\Todo: do we need $\Phi$ to preserve the metric?) 
For a function $\psi \in C^{\infty}_c(\Rnz)$, consider the operator $\Phi^*_{\psi}$ on $\Gamma^{\infty}(E)$ defined by
\beq{eq def Phi psi}
(\Phi^*_{\psi}s)(m) := \int_{\R} \psi(t) \Phi_{-t} (s(\varphi_t(m)))\, dt,
\eeq
for $s \in \Gamma^{\infty}(E)$ and $m \in M$. 

Let $A \in \End(E)$ be $G$-equivariant, and suppose that $A$ commutes with $\Phi_t$ for all $t \in \R$.
% In the application to the Ruelle dynamical $\zeta$-function, we will take $A = (-1)^N N$, with $N$ the number operator on differential forms. 
We will obtain a criterion for the operator $A\Phi^*_{\psi}$ to be flat $g$-trace class for $g$-nondegenerate flows and $\psi \in C^{\infty}_c(\Rnz)$, see Corollary \ref{cor g anosov trace cl} and Theorem \ref{prop fixed pt gen}. In that case
we write
\beq{eq flat g trace flow}
\langle \Tr_g^{\flat}(A\Phi^*), \psi \rangle := \Tr_g^{\flat}(A \Phi^*_{\psi}).
\eeq
We will see in Theorem \ref{prop fixed pt gen} that this defines a distribution on $\Rnz$.

Fix $x \in G$, and let $x\cdot AK \in \Gamma^{\infty}_c\bigl (\Rnz \times M \times M, \Rnz \times E^* \boxtimes E \bigr)'$ be defined by
\beq{eq def x K}
\langle x  \cdot AK, \psi \otimes \xi \otimes s\rangle := \int_M \langle \xi(m), (x\cdot A\Phi^*_{\psi}s)(m) \rangle\, dm,
\eeq
for $\psi \in C^{\infty}_c(\Rnz)$, $\xi \in \Gamma^{\infty}_c(E^*)$ and $s\in \Gamma^{\infty}_c(E)$.

\begin{proposition}\label{prop restr K}
Suppose that the flow $\varphi$ is $g$-nondegenerate. Then
for all $h \in G$, 
\[
\WF(hgh^{-1}\cdot AK) \cap \cN^*(\R\setminus\{0\} \times \Delta(M)) = \emptyset.
\]
%Then $hgh^{-1}\cdot AK$ has a well-defined distributional restriction to $\Rnz \times \Delta(M)$ for all $h \in G$.
\end{proposition}
The proposition is proved in Subsection \ref{sec flow WF}.

If $\varphi$ is $g$-nondegenerate, then
Proposition \ref{prop restr K} implies that
 $hgh^{-1}\cdot AK$ has a well-defined restriction to $\Rnz \times \Delta(M)$ for all $h \in G$. 
Analogously to Definition \ref{def fibre trace}, we  define 
\[
\tr \left( hgh^{-1}\cdot AK|_{\Rnz \times \Delta(M)}\right) \in \calD'(\Rnz \times \Delta(M))
\]
by
\[
\left\langle \tr \left( hgh^{-1}\cdot AK|_{\Rnz \times \Delta(M)}\right) , \psi \otimes f  \right\rangle := 
\left\langle hgh^{-1}\cdot AK|_{\Rnz \times \Delta(M)} , \psi \otimes f \Id_E \right\rangle, 
\]
for $\psi \in C^{\infty}_c(\Rnz)$ and $f \in C^{\infty}_c(\Delta(M))$.

%Hence we obtain the following corollary.

Proposition \ref{prop restr K} has the following consequence. 
\begin{corollary} \label{cor g anosov trace cl} % \label{def flat g trace flows}
Suppose that the flow $\varphi$ is $g$-nondegenerate, and that the integral
\[
\int_{G/Z} \bigl\langle \tr\bigl(hgh^{-1} \cdot AK|_{\Rnz \times \Delta(M)}\bigr), \psi \otimes (p_2^*f)|_{\Delta(M)} \bigr\rangle\, d(hZ)
\]
converges absolutely for  $\psi \in C^{\infty}_c(\Rnz)$ and $f \in C^{\infty}_c(M)$. Then $A\Phi^*_{\psi}$ is flat $g$-trace class for all such $\psi$, and \eqref{eq flat g trace flow} equals
\[
\langle
\Tr_g^{\flat}(A\Phi^*), \psi \rangle = \int_{G/Z} \bigl\langle \tr\bigl(hgh^{-1} \cdot AK|_{\Rnz \times \Delta(M)}\bigr), \psi \otimes (p_2^*\chi)|_{\Delta(M)} \bigr\rangle\, d(hZ).
\]
\end{corollary}
%We will see in Proposition \ref{prop fixed pt gen} that $\Tr^{\flat}_g(A\Phi^*)$ is a well-defined distribution on $\Rnz$, if $\varphi$ is $g$-nondegenerate. This distribution does not depend on the choice of $\chi$, analogously to Lemma \ref{lem flat g trace indep chi}. 

\subsection{Relation with the classical flat trace}

In Example \ref{ex class flat trace}, we saw that for the trivial group, the 
 flat $g$-trace of Definition \ref{def flat g trace} reduces to the classical flat trace.
However, there is also a more refined relation between the flat $g$-trace and the classical flat trace, in terms of universal covers. 
 %This extends to the trace for flows of Definition \ref{def flat g trace flows}.

We consider the setting of Example \ref{ex lift univ cover}. 
Suppose that $M$ is the universal cover of a compact manifold $X$, and $G = \Gamma = \pi_1(X)$. Let $\varphi_X$ be a flow on $X$, and  $\Phi_X$ a fibre-wise linear lift of $\varphi_X$ to $E_X$.
Let  $\varphi$ be the lift of $\varphi_X$ to $M$, and let $\Phi$ be the $\Gamma$-equivariant lift of $\Phi_X$ to $E$.

 Let $u_X$ be the generating vector field of $\varphi_X$, and let $u$ be the generating vector field of $\varphi$.
 
Let $\Phi_X$ be the flow on $E_X = E/\Gamma \to X$ induced by a $\Gamma$-equivariant lift $\Phi$ of $\varphi$ to $E$, and $K_X$ the corresponding Schwartz kernel. 
 Let $q_M\colon M \to X$ be the covering map. We use similar notation for the covering maps between other spaces, such as $q_{\Rnz \times  M \times M}\colon \Rnz \times  M \times M \to \Rnz \times  X \times X$. 

\begin{lemma}\label{lem nondeg M X}
Let $t \in \Rnz$. Then the following two conditions are equivalent.
\begin{enumerate}
\item
For all $x \in X$ such that $(\varphi_X)_t(x)=x$, we have 
on $T_xX/ \R u_X(x)$ that
\beq{eq det X}
\det\bigl(T_{x}(\varphi_X)_t - 1 \bigr)\not=0.
\eeq
\item
For all $\gamma \in \Gamma$, and all $m \in M$ such that $\varphi_t(m) = \gamma m$, we have 
on $T_mM/\R u(m)$ that
\beq{eq det M}
\det\bigl(T_{m}( \varphi_t\circ \gamma^{-1}) - 1 \bigr)\not=0.
\eeq
\end{enumerate}
\end{lemma}
\begin{proof}
Let $t \in \Rnz$, $m \in M$ and $\gamma \in \Gamma$ be such that $ \varphi_t(m) = \gamma m$. Write $x:= \Gamma m$. The  the equality $u_X(x) = T_mq(u(m))$ and commutativity of the  diagram below imply that the determinants \eqref{eq det X} and $\eqref{eq det M}$ are equal:
\[
\xymatrix{
T_mM \ar[r]^-{T_m \gamma^{-1}} \ar[d]_-{T_mq_M}^-{\cong}& T_{\gamma^{-1}m}M \ar[r]^-{T_{\gamma^{-1}m}\varphi_t} & T_mM  \ar[d]_-{T_mq_M}^-{\cong}\\
T_xX \ar[rr]^-{T_x (\varphi_X)_t} & & T_xX.
}
\]
\end{proof}

\begin{proposition}\label{prop KX}
Under the equivalent conditions of Lemma \ref{lem nondeg M X},
\[
\Tr^{\flat}(\Phi_X) = \sum_{(\gamma)}\Tr^{\flat}_{\gamma}(\Phi),
\]
where the sum is over the conjugacy classes in $\Gamma$. In particular, the sum on the right converges in $\calD'(\Rnz)$.
\end{proposition}
This proposition will be proved in Subsection \ref{sec class flat trace}

\subsection{The equivariant Guillemin trace formula}

As in Subsection \ref{sec equivar flows}, we consider an equivariant flow $\varphi$ on $M$, generated by a nonvanishing vector field $u$.

\begin{definition}\label{def length spectrum}
Let $x \in G$.
\begin{itemize}
\item[(a)]
 The \emph{$x$-delocalised length spectrum} of $\varphi$ is
\beq{eq def Lg}
L_{x}(\varphi) := \{l \in \Rnz; \exists m \in M: \varphi_l(m) = x m\}.
\eeq
\end{itemize}
Let
 $l \in L_{x}(\varphi)$.
 \begin{itemize}
 \item[(b)]
A \emph{flow curve}  of $\varphi$ is a curve $\gamma\colon \R \to M$ such that $\gamma(t) = \varphi_t(\gamma(0))$ for all $t \in \R$. A flow curve $\gamma$ is said to be
  \emph{$(x, l)$-periodic} if $\gamma(l) = x \gamma(0)$. 
 There is an action by $\R$ on the set $\tilde \Gamma^{x}_l(\varphi)$ of such curves, defined by
 \beq{eq action R Gamma}
( s\cdot \gamma)(t) = \gamma(t+s),
 \eeq
 for $\gamma \in \tilde \Gamma^{x}_l(\varphi)$ and $s,t \in \R$. We write
 \beq{eq def Gamma x l}
 \Gamma^{x}_l(\varphi) := \tilde \Gamma^{x}_l(\varphi)/\R.
 \eeq
 We will often tacitly choose representatives of classes in $\Gamma^{x}_l(\varphi)$, and implicitly claim that what follows is independent of the choice of this representative; compare Remark \ref{rem indep gamma} below.
 \end{itemize}
 
 Let
 $\gamma \in \Gamma_l^{x}(\varphi)$.
 \begin{itemize}
 \item[(c)]
% $m_0 \in M$ satisfies $ \varphi_l(m_0) = g m_0$, 
The  \emph{linearised $x$-delocalised Poincar\'e map}  of $\gamma$ is the map
\beq{eq def Poincare}
P^{x}_{\gamma} \colon T_{\gamma(0)}M/\R u(\gamma(0)) \to T_{\gamma(0)}M/\R u(\gamma(0))
\eeq
induced by $T_{\gamma(0)} (\varphi_{l}\circ x^{-1})$. 
\item[(d)] The \emph{$\chi$-primitive period} of $\gamma$ is
\beq{eq def Tgamma}
T_{\gamma} := \int_{I_{\gamma}} \chi(\gamma(s))\, ds
\eeq
(when convergent), 
where $I_{\gamma}$ is an interval such that $\gamma|_{I_{\gamma}}$ is a bijection onto the image of $\gamma$, modulo sets of measure zero. See Lemma \ref{lem conv Tgamma} for a sufficient condition for convergence of the integral.
\end{itemize}
\end{definition}

%The number $T_{\gamma}$ may depend on $\chi$ in general.

%\begin{remark}
%In the setting of $\eqref{eq def Poincare}$, the map  $T_{\gamma(0)} (\varphi_{l}\circ x^{-1})$ indeed preserves the subspace $T_{\gamma(0)}M/\R u(\gamma(0)) = \gamma'(0)^{\perp} \subset T_{\gamma(0)}M$. To see this, let $v \in \gamma'(0)^{\perp}$. Then, because the group action and the map $T_{\gamma_0}\varphi_l$ preserve the Riemannian metric, 
%\beq{eq inner prod P}
%\bigl(T_{\gamma(0)} (\varphi_l \circ x^{-1})v, \gamma'(0) \bigr)_{T_{\gamma(0)}M} = 
%\bigl (v, T_{\gamma(0)} (x\circ \varphi_{-l})  \gamma'(0) \bigr)_{T_{\gamma(0)}M}. 
%\eeq
%And because the flow is equivariant,
%\[
%T_{\gamma(0)} (x\circ \varphi_{-l})  \gamma'(0) = \ddt \varphi_t (x\circ \varphi_{-l}(\gamma(0))) = \gamma'(0), 
%\]
%because $x\circ \varphi_{-l}(\gamma(0)) = \gamma(0)$. 
%So \eqref{eq inner prod P} is zero.
%\end{remark}

\begin{remark}
As the domain of a flow curve is taken to be the real line, a single  flow curve $\gamma$ can occur in more than one of the sets $\Gamma_l^x(\varphi)$, if $\gamma(l) = x\gamma(0)$ for more than one combination of $l\in \Rnz$ and $x \in G$. 
% See the example of the circle in Subsection \ref{sec ex circle}, where is only one flow curve modulo the action \eqref{eq action R Gamma}, which lies in all sets $\Gamma_l^g(\varphi)$ that occur.
%
For a flow curve $\gamma \in \Gamma_l^g(\varphi)$, the maps $P^x_{\gamma}$ depend on $l$, even though this is not reflected in the notation.
\end{remark}

\begin{example}\label{ex Thgam ZR}
Let $M = G = \R$, and consider the flow $\varphi_t(m) = m+t$. Then the only flow curve of $\varphi$, modulo the action \eqref{eq action R Gamma}, is $\gamma(t) = t$.
Let $x \in \Rnz$. Then $L_x(\varphi) = \{x\}$, and $\Gamma_x^x(\varphi) = \{\gamma\}$. Here we see that it is relevant to allow negative numbers in $L_x(\varphi)$. Now  $I_{\gamma} = \R$,  so
\[
T_{\gamma} = \int_{\R}\chi(s)\, ds = 1.
\]
\end{example}

\begin{remark}
It is possible that the $x$-delocalised length spectrum \eqref{eq def Lg} is empty for some $x \in G$, but not for others. It is empty when $g=e$ in Example \ref{ex Thgam ZR}, for example. 
\end{remark}

\begin{example}\label{ex prim per e}
If $G$ is compact (for example, trivial), then we may take $\chi$ to be the constant function $1$. 
If $\gamma$ is periodic (see part (b) of Lemma \ref{lem conv Tgamma} for a sufficient condition), then $I_{\gamma}$ is bounded. 
Then $T_{\gamma}$ is the length of  $I_{\gamma}$, which is the primitive period 
\[
T_{\gamma}^{\#} = \min\{t>0; \gamma(t) = \gamma(0)\} 
\]
of $\gamma$.
\end{example}

\begin{example}
For the geodesic flow on the universal cover of a negatively curved manifold, the $\chi$-primitive period of a geodesic equals its primitive period in the usual sense, for well-chosen $\chi$. (See the proof of Lemma 3.19 in \cite{HS25b}.)
\end{example}

Recall that we fixed $g \in G$. It follows from equivariance of $\varphi$ that
$T_m (\varphi_l \circ g^{-1})(u(m)) = u(m)$ for all $m \in M$ such that $\varphi_l(m) = gm$. This
implies that $T_{m}(\varphi_{l}\circ g^{-1}) - 1$ descends to a linear endomorphism of $T_{m}M/\R u(m)$, which we denote by $ (T_{m}(\varphi_{l}\circ g^{-1}) - 1)|_{T_{m}M/\R u(m)}$. 
\begin{definition}\label{def nondeg}
The flow $\varphi$ is \emph{$g$-nondegenerate} if for all $l \in \Rnz$ and   $m \in M$ such that $ \varphi_l(m) = gm$, 
\beq{eq flow nondeg}
\det\bigl( (T_{m}(\varphi_{l}\circ g^{-1}) - 1)|_{T_{m}M/\R u(m)} \bigr)\not=0.
\eeq
\end{definition}
In other words, $\varphi$ is $g$-nondegenerate if the map $P_{\gamma}^g - 1_{T_{\gamma(0)}M/\R u(\gamma(0))}$ is invertible for all $l\in L_g(\varphi)$ and $\gamma \in \Gamma_l^g(\varphi)$.
%
% all Poincar\'e maps \eqref{eq def Poincare} are invertible. 
% For $g=e$, this condition is true for Anosov flows. 
%We will use the fact that \eqref{eq flow nondeg} is equivalent to $\det\bigl( (T_{m}(g^{-1} \circ \varphi_{t}) - 1)|_{T_{m}M/\R u(m)} \bigr)\not=0$.

\begin{remark}
If $g=e$, then an Anosov flow is $e$-nondegenerate, but the latter condition is weaker. (The Anosov condition involves all flow curves, whereas $e$-nondegeneracy is only about closed curves.) In particular, the $g$-nondegeneracy condition does not require the choice of a Riemannian metric. 
\end{remark}

The following lemma generalises a well-known criterion for fixed points of group actions, or flow curves of a given flow, to be isolated. See page 377 of \cite{ABI} for a comment on the non-equivariant case. %The proof is a combination of standard arguments.
\begin{lemma}\label{lem Gammagl discrete}
Suppose that  $\varphi$ is {$g$-nondegenerate} and
let  $l \in L_{g}(\varphi)$. 
%
%Suppose that $\varphi$ satisfies the condition of Proposition \ref{prop restr K}.
%For all $l \in L_{g}(\varphi)$,
Then the set $\Gamma^{g}_l(\varphi)$ is discrete, in the sense that it is countable, and the images of the curves in $\Gamma^{g}_l(\varphi)$ together form the closed subset  $M^{g^{-1}\varphi_l} \subset M$ of points fixed by $g^{-1} \circ \varphi_l$.
\end{lemma}
\begin{proof}
%Let $M^{g^{-1} \circ \varphi_l}$ be the fixed point set of $g \circ \varphi$. 
Since $l \in L_g(\varphi)$, the fixed-point set  $M^{g^{-1} \circ \varphi_l}$ is nonempty.
Let $m \in M^{g^{-1} \circ \varphi_l}$.
%
%Consider the diagonal embedding $\Delta\colon M \to M \times M$. We denote its derivative by $\Delta$ as well.
A vector $v \in T_mM$ is tangent to $M^{g^{-1} \circ \varphi_l}$ if and only if 
%\[
%\Delta(v) \in T_{m,m}\Delta(M) \cap T_{m,m} \graph(g^{-1} \circ \varphi_l).
%\]
%The latter condition is equivalent to
 $v \in \ker(T_{m}(g^{-1}\circ \varphi_l) - 1) = \R u(m)$. So
\[
T_m M^{g^{-1} \circ \varphi_l} = \R u(m).
\]
which implies that $M^{g^{-1} \circ \varphi_l}$ is one-dimensional.

Let $\gamma(t) = \varphi_t(m)$ be the flow curve through $m$. This curve lies in the one-dimensional submanifold $M^{g^{-1} \circ \varphi_l}$, and hence its image equals the connected component of  $M^{g^{-1}\circ \varphi_l}$ containing $m$. We find that the images of all flow curves in $\Gamma^g_l(\varphi)$ are connected components of the one-dimensional manifold $M^{g^{-1} \circ \varphi_l}$.
%, which implies the claim.
\end{proof}

\begin{lemma}\label{lem conv Tgamma}
Suppose that $\varphi$ is $g$-nondegenerate. Let $l \in L_g(\varphi)$ and $\gamma \in \Gamma^g_l(\varphi)$.
\begin{enumerate}
\item[(a)] The expression \eqref{eq def Tgamma} for $T_{\gamma}$ converges.
\item[(b)] If $M$ is compact, then $\gamma$ is periodic.
\item[(c)] If $g$ lies in a compact subgroup of $G$, then $\gamma$ is periodic.
\end{enumerate}
\end{lemma}
\begin{proof}
By Lemma \ref{lem Gammagl discrete}, the image of $\gamma$ is closed. So $\im(\gamma) \cap \supp(\chi)$ is compact. Taking $I_{\gamma}$ to be closed, as we may, and using properness of  $\gamma|_{I_{\gamma}}$, we find that $\gamma^{-1}(\supp(\chi)) \cap I_{\gamma}$ is compact. So the integrand in \eqref{eq def Tgamma} has compact support, and (a) follows.

If $M$ is compact, then the image of $\gamma$ is compact, because it is closed. Hence $\gamma$ is periodic, and (b) follows.

For all $n \in \Z$ and $t \in [0,l]$, we have
\[
\gamma(t+nl) = g^n \gamma(t) \subset \overline{g^{\Z}} \gamma([0,l]).
\]
If $g$ lies in a compact subgroup of $G$, then the set on the right is compact. Because the image of $\gamma$ is closed and contained in this set, it is compact. Hence (c) follows.
\end{proof}

It follows from the definitions that for all $h \in G$, 
\beq{eq Lhghinv}
L_{hgh^{-1}}(\varphi) = L_g(\varphi)
\eeq 
 and that for $l$ in this set, 
 \beq{eq Gamma hghinv}
 \Gamma_l^{hgh^{-1}}(\varphi) = h\cdot \Gamma_l^g(\varphi). 
 \eeq
 So by Lemma \ref{lem Gammagl discrete}, the set $\Gamma_l^{hgh^{-1}}(\varphi) $ is discrete for all $l \in L_g(\varphi)$ if $\varphi$ is $g$-nondegenerate.

%Fix $g \in G$ as before. 
Let $A \in \End(E)$ be $G$-equivariant, and suppose that  it commutes with the flow $\Phi$.  
For $x \in G$, let $L_{x}(\varphi)$ and $\Gamma^{x}_l(\varphi)$, for $l \in L_{x}(\varphi)$, be as in Definition \ref{def length spectrum}.

\begin{lemma}
Let $x \in G$, 
 $l \in L_{x}(\varphi)$ and $\gamma \in \Gamma^{x}_l(\varphi)$. Then the number
\beq{eq trace g Phi}
\tr(A x(\Phi_{-l})_{\gamma(t)})
\eeq
is independent of $t$.
\end{lemma}
\begin{proof}
For all $t$, we have a commuting diagram
\[
\xymatrix{
E_{\gamma(0)} \ar[d]_-{(\Phi_t)_{\gamma(0)}} \ar[rrr]^-{A x (\Phi_{-l})_{\gamma(0)}} &&& E_{\gamma(0)}\ar[d]^-{(\Phi_t)_{\gamma(0)}} \\
E_{\gamma(t)} \ar[rrr]^-{Ax (\Phi_{-l})_{\gamma(t)}} &&& E_{\gamma(t)}.
}
\]
Here we use the assumption that $A$ commutes with $x$ and $\Phi_t$. 
We find that $\tr(Ax(\Phi_{-l})_{\gamma(t)}) = \tr(Ax(\Phi_{-l})_{\gamma(0)})$.
\end{proof}
We denote the number \eqref{eq trace g Phi} by $\tr(Ax\Phi_{-l}|_\gamma)$.

Recall that we have chosen $g \in G$ such that $G/Z$ has a $G$-invariant measure $d(hZ)$. We denote the Dirac $\delta$ distribution  at $l \in \R$ by $\delta_l$. 
%Recall Definition \ref{def nondeg} of the $g$-nondegenerate property. 
\begin{theorem}[Equivariant Guillemin trace formula]\label{prop fixed pt gen}
Suppose that the flow $\varphi$ is $g$-nondegenerate. Then the operator $A \Phi^*_{\psi}$ is flat $g$-trace class for all $\psi \in C^{\infty}_c(\Rnz)$. Furthermore, \eqref{eq flat g trace flow} defines a distribution on $\Rnz$, equal to
\beq{eq fixed pt Trb 2}
\Tr^{\flat}_g(A \Phi^*) = \int_{G/Z}  \sum_{l \in L_{g}(\varphi)} \sum_{\gamma \in \Gamma^{hgh^{-1}}_l(\varphi)} 
\frac{\tr(A hgh^{-1}\Phi_{-l}|_\gamma) }{|\det(1-P_{\gamma}^{hgh^{-1}})  |} T_{\gamma} \delta_{l}  \, d(hZ).
\eeq
\end{theorem}

\begin{remark} \label{rem indep gamma}
The number $\det(1-P_{\gamma}^{hgh^{-1}})$  in \eqref{eq fixed pt Trb 2} is independent of the choice of the representative $\gamma$ of a class in $\Gamma_l^{hgh^{-1}}(\varphi)$. Indeed, suppose that $ x\in G$,  $l \in L_x(\varphi)$ and  that $\gamma_1, \gamma_2 \in \tilde \Gamma_l^x(\varphi)$ are in the same $\R$-orbit for the action \eqref{eq action R Gamma}. Then
\[
P_{\gamma_2}^x = \varphi_s \circ P_{\gamma_1}^x\circ \varphi_{-s}.
\]
\end{remark}

Theorem \ref{prop fixed pt gen} will be used to express the equivariant Ruelle $\zeta$-function in terms of the flat $g$-trace in \cite{HS25b}.

\section{Proofs of properties of the flat $g$-trace}

This section contains proofs of Propositions \ref{prop flat g trace indep chi}, \ref{prop restr K} and \ref{prop KX}. We also prove Lemma \ref{lem flat trace trivial}, which is used in the proof of Proposition \ref{prop fixed pt Trb}.

\subsection{Independence of the cutoff function}\label{sec indep chi}

 If $x \in G$, then we write
\beq{eq def Delta x}
\Delta_x(M) := \{(xm,m); m \in M\} \subset M \times M.
\eeq
\begin{lemma}\label{lem pullback cts Delta}
For all $x \in G$, the map $\kappa \mapsto x \cdot \kappa$ is a continuous linear map from 
\beq{eq gen sect x}
\{\kappa \in \Gamma^{-\infty}(E \boxtimes E^*); \WF(\kappa) \cap \cN^*(\Delta_{x^{-1}}(M)) = \emptyset\}
\eeq
to
\beq{eq gen sect e}
\{\kappa \in \Gamma^{-\infty}(E \boxtimes E^*); \WF(\kappa) \cap \cN^*(\Delta(M)) = \emptyset\},
\eeq
with respect to the topologies on these spaces as in Definition 8.2.2 in \cite{Hormander1}.
\end{lemma}
\begin{proof}
The set
\[
(x^{-1} \times \Id_M)^* \cN^*(\Delta_{x^{-1}} (M)) := 
\{  T(x^{-1} \times \Id_M)^*\xi; \xi \in   \cN^*(\Delta_{x^{-1}} (M))\}
\]
equals $ \cN^*(\Delta (M))$, by a direct verification. Furthermore, if $\kappa \in \Gamma^{-\infty}(E \boxtimes E^*)$, then $x\cdot \kappa = (x^{-1} \times \Id_M)^*\kappa$. % (pullbacks along diffeomorphisms are always defined). 
 Hence the claim follows from standard results on pulling back distributions; see Theorem 8.2.4 in \cite{Hormander1}.
\end{proof}

The proof of Proposition \ref{prop flat g trace indep chi}  is based on Lemmas \ref{lem Trg shift} and \ref{lem psi u f}.
\begin{lemma}\label{lem Trg shift}
For all $\kappa$ as in Definition \ref{def flat g trace},  $f \in C^{\infty}_c(M)$ and all $x,h \in G$,
\beq{eq Trg shift}
\langle \tr(hgh^{-1} \cdot \kappa|_{\Delta(M)}),  (p_2^*(x\cdot f))|_{\Delta(M)}\rangle = 
\langle \tr(x^{-1}hgh^{-1}x \cdot \kappa|_{\Delta(M)}),  (p_2^* f)|_{\Delta(M)}\rangle.
\eeq
\end{lemma}
\begin{proof}
If $\kappa$ is smooth, then by Example \ref{ex fibre trace smooth}, a substitution $m \mapsto xm$, equivariance of $T$ and the trace property of the fibre-wise trace, the left hand side of \eqref{eq Trg shift} equals
\begin{multline*}
\int_{M} f(x^{-1}m) \tr(hgh^{-1} \kappa(hg^{-1}h^{-1}m,m))\, dm \\
= 
\int_{M} f(m) \tr\bigl(x^{-1}hgh^{-1}x \kappa(x^{-1}hg^{-1}h^{-1}xm,m) \bigr)\, dm.
\end{multline*}
The expression on the right is the right hand side of \eqref{eq Trg shift}, so the equality holds for smooth kernels $\kappa$. 

The case for general $\kappa$ follows because the smooth kernels are dense in the space of generalised sections with wave front sets disjoint from $\cN^*(\Delta_y(M))$, for $y \in G$ (see Theorem 8.2.3 in \cite{Hormander1}), and both sides of \eqref{eq Trg shift} depend continuously on $\kappa$. For this continuity, note that the distribution $y\cdot \kappa$ in \eqref{eq gen sect e} depends continuously on $\kappa$ in \eqref{eq gen sect x}, by Lemma \ref{lem pullback cts Delta}. Restriction to $\Delta(M)$ is continuous on the space \eqref{eq gen sect e}, and the trace in Definition \ref{def fibre trace} is continuous from $\Gamma^{-\infty}(E^*\boxtimes E|_{\Delta(M)})$ to $\calD'(\Delta(M))$. So $\kappa \mapsto  \tr(y \cdot \kappa|_{\Delta(M)})$ is a continuous  map from \eqref{eq gen sect x} to $\calD'(\Delta(M))$. Evaluation of distributions on test functions is continuous by definition, so both sides of \eqref{eq Trg shift} indeed depend continuously on $\kappa$ in  \eqref{eq gen sect x}. And the condition \eqref{eq condition WF K} is equivalent to $\kappa$ lying in   \eqref{eq gen sect x}, for $x = hgh^{-1}$.
\end{proof}

\begin{lemma}\label{lem psi u f}
Let $\eta$ be a $G$-invariant distribution on $M$. Let $f \in C^{\infty}_c(M)$ be such that for all $m \in M$, 
\[
\int_G f(xm)\, dx = 0.
\]
Let $\psi \in C^{\infty}_c(G)$. Then
\beq{eq int psi u f}
\int_G \psi(x) \langle \eta, x\cdot f \rangle\, dx = 0.
\eeq
\end{lemma}
\begin{proof}
We prove the claim for $\eta \in C^{\infty}(M)^G$, the general case follows because the left hand side of \eqref{eq int psi u f} equals 
\[
\left\langle \eta, \int_G \psi(x)  x\cdot f \, dx  \right\rangle
\]
and hence 
depends continuously on $\eta$.

Suppose  that $M = G \times_K N$ for a compact subgroup $K<G$ and a compact, $K$-invariant submanifold $N\subset M$. The general case follows from an application of Palais' slice theorem \cite{Palais61} and the use of a partition of unity. Since $M/G$ is compact, it can be covered by finitely many open sets of the form $G \times_K N$. A sum over these sets commutes with the integrals that occur in the argument below.

Choose the density $dm$ so that for all $a \in C_c(M)$, 
\[
\int_M a(m)\, dm = \int_G \int_N a(gn)\, dn\, dg
\]
for a $K$-invariant density $dn$ on $N$. 
Then 
 the left hand side of \eqref{eq int psi u f} equals
\beq{eq int psi u f 2}
\int_G \psi(x) \int_G \int_N \eta(n) f(x^{-1}yn)\, dn\, dy\, dx.
\eeq
Since $\psi$ and $f$ have compact supports, the action is proper, and $N$ is compact, the integrand has compact support. So we may freely interchange the order of integration. 
Interchanging the inner integrals over $G$ and $N$, and using left invariance of $dy$,  
we find that \eqref{eq int psi u f 2} equals
\[
\int_G \psi(x) \int_N \eta(n)  \int_G f(yn)\, dy \, dn\, dx = 0.
\]
\end{proof}

\begin{remark}
The proof of Lemma \ref{lem psi u f} can also be applied directly to $G$-invariant distributions $\eta$, rather than smooth functions. Then one uses the fact that such a distribution has a well-defined restriction to $N$, for example in \eqref{eq int psi u f 2}, because it is smooth on $G$-orbits. To see the latter, note that for any $\varphi \in C^{\infty}_c(G)$ with integral $1$, 
\[
\eta = \int_G \varphi(x) \, x\cdot \eta\, dx.
\]
And the right hand side is smooth on $G$-orbits.
\end{remark}

%\Todo: approximation by smooth functions in previous two lemmas is ok?

\begin{proof}[Proof of Proposition \ref{prop flat g trace indep chi}]
Let $f \in C^{\infty}_c(M)$, and suppose that for all $m \in M$, 
\[
\int_G f(xm)\, dx = 0.
\]
We will show that the integral \eqref{eq def flat g trace 1}  equals zero; this implies the claim.

For $x \in G$, let $x\kappa x^{-1}$ be the Schwartz kernel of $x\circ T \circ x^{-1}$. By $G$-equivariance of $T$, we have $x\kappa x^{-1} = \kappa$. Let $\psi \in C^{\infty}_c(G)$ be a function with integral equal to one. Then \eqref{eq def flat g trace 1}  equals
\begin{multline}\label{eq trace indep chi}
\int_{G/Z} \int_G \psi(x) \langle \tr(hgh^{-1} \cdot x\kappa x^{-1}|_{\Delta(M)}),  (p_2^*f)|_{\Delta(M)}\rangle \, dx \, d(hZ)\\
 = 
\int_{G/Z} \int_G \psi(x) \langle \tr(x^{-1}hgh^{-1}x\kappa|_{\Delta(M)}),  (p_2^*f)|_{\Delta(M)}\rangle \, dx \, d(hZ) \\
 = 
\int_{G/Z} \int_G \psi(x) \langle \tr(hgh^{-1}\kappa|_{\Delta(M)}),  (p_2^*(x \cdot f))|_{\Delta(M)}\rangle \, dx \, d(hZ).
\end{multline}
%(\Todo: first equality is ok?)
In the second equality, we used Lemma \ref{lem Trg shift}. Since the integral \eqref{eq def flat g trace 1} converges absolutely and $\psi$ has compact  support, the above integrals  converge absolutely. Hence the right hand side of \eqref{eq trace indep chi}  equals
\beq{eq indep chi 2}
\int_G \psi(x) \int_{G/Z}  \langle \tr(hgh^{-1}\kappa|_{\Delta(M)}),  (p_2^*(x \cdot f))|_{\Delta(M)}\rangle  \, d(hZ)\, dx.
\eeq
The distribution
\[
\varphi \mapsto \int_{G/Z}  \langle \tr(hgh^{-1}\kappa|_{\Delta(M)}),  (p_2^*\varphi)|_{\Delta(M)}\rangle  \, d(hZ)
\]
on $M$ is $G$-invariant by Lemma \ref{lem Trg shift} and $G$-invariance of $d(hZ)$. Lemma \ref{lem psi u f} then implies that
\eqref{eq indep chi 2} equals zero and the proof is complete.
\end{proof}

\subsection{The flat $g$-trace for equivariant flows} \label{sec flow WF}

We now prove Proposition \ref{prop restr K}, which states that the condition \eqref{eq condition WF K} holds for the operator $A \Phi_{\psi}^*$ if $\varphi$ is $g$-nondegenerate.
Consider the map
\[
\Psi_x\colon (\Rnz) \times M \times M \to M \times M
\]
given by
\[
\Psi_x(t,m,m') = (x\varphi_t(m),m').
\]
\begin{lemma}\label{lem TPhi}
The derivative of $\Psi_x$ at $(t,m,m')$ is given by
\[
T_{(t,m,m')}\Psi_x(s,v,v') = \bigl( su(x \varphi_t(m)) +T_m(x\varphi_t)(v), v' \bigr),
\]
for all $s \in T_t\R = \R$, $v \in T_mM$ and $v' \in T_{m'}M$.
\end{lemma}
\begin{proof}
This is a direct computation.
\end{proof}

Consider the set
%\begin{multline}\label{eq def S}
\beq{eq def S}
S_x:= \Psi_x^{-1}(\Delta(M)) 
= \{(t,m,x\varphi_t(m)) \in (\Rnz) \times M \times M; t \in \R, m \in M\}.
\eeq
%\end{multline}
\begin{lemma}\label{lem TS}
The set $S_x$ is a smooth submanifold, and its tangent bundle is
\beq{eq TS}
TS_x = \{\bigl( (t,s), v,  su(x\varphi_t(m)) +T_m(x\varphi_t)(v)\bigr); s,t \in \R, m \in M, v \in T_mM\}.
\eeq
\end{lemma}
\begin{proof}
If $t \in \Rnz$ and $m \in M$, then Lemma \ref{lem TPhi} implies that the image of $T_{(t, m, x\varphi_t(m))}\Psi_x$ contains $\{0\} \times T_{x\varphi_t(m)}M$. So the sum of this image and 
\[
T_{\Psi_x(t, m, x\varphi_t(m))} \Delta(M) = \{(v,v); v \in T_{x\varphi_t(m)}M\}
\]
is the whole tangent space $T_{\Psi_x(t, m, x\varphi_t(m))} (M \times M)$. This implies that $\Psi_x$ is transversal to $\Delta(M)$, and that $S_x$ is a smooth submanifold of $\Rnz \times M \times M$.

The tangent bundle to $TS_x$ is
\[
TS_x = (T\Psi_x)^{-1}(T(\Delta(M))).
\]
By  Lemma \ref{lem TPhi}, this equals the right hand side of \eqref{eq TS}.
\end{proof}

The following lemma is used on page 24 of \cite{DZ16} for the identity element in $G$, where $V(x)$ should be replaced by $V(\varphi_t(x))$.
\begin{lemma}\label{lem NS}
The conormal bundle of $S_x$ is
\[
\cN^*(S_x) = 
\{\bigl( (t, \langle \zeta, u(m)\rangle) , \zeta, -T_m(\varphi_{-t}x^{-1})^*\zeta)\bigr);  t \in \R\setminus \{0\}, \zeta \in T_{m}M\}.
%\{\bigl( (t,-\langle \xi, u(\varphi_t(m)\rangle) , -T_m\varphi_t^*\xi, \xi)\bigr);  t \in \R\setminus \{0\}, \xi \in T_{\varphi_t(m)}M\}.
\]
\end{lemma}
\begin{proof}%[Proof of Lemma \ref{lem NS}.]
Let $a \in \R = T^*_t\R$, $\zeta \in T^*_mM$ and $\xi \in T^*_{x\varphi_t(m)}M$. By Lemma \ref{lem TS}, we have $(a,\zeta, \xi) \in \cN^*(S_x)$ if and only if for all  $s \in \R$ and $v \in T_mM$,
\[
\begin{split}
0 &=\bigl \langle(a,\zeta, \xi) ,  \bigl( (t,s), v,  su(x\varphi_t(m)) +T_m(x\varphi_t)(v)\bigr)\bigr\rangle\\
&=s(a+   \langle \xi, u(x\varphi_t(m)) \rangle) + \langle\zeta, v \rangle
	+\langle \xi, T_m(x\varphi_t)(v) \rangle.
\end{split}
\]
So
\[
\begin{split}
a+   \langle \xi, u(x\varphi_t(m)) \rangle &= 0\\
\zeta
+(T_m(x\varphi_t))^* \xi&=0.
\end{split}
\]
The second equality implies that $\xi = -(T_m\varphi_{-t}x^{-1})^*\zeta$, and then the first becomes
\[
a =  \langle (T_m\varphi_{-t}x^{-1})^*\zeta, u(x\varphi_t(m)) \rangle = \langle \zeta, u(m)\rangle.
\]
In the last equality, we used the fact that 
\[
u(x\varphi_t(m)) = T_m (x\varphi_t )(u(m)).
\]
\end{proof}

\begin{lemma}\label{lem NS NDelta}
Suppose that the flow $\varphi$ is $x$-nondegenerate. 
Then 
\[
\cN^*(S_{x^{-1}}) \cap \cN^*(\R\setminus\{0\} \times \Delta(M)) = \bigl( \Rnz \times \Delta(M^{x^{-1}\circ \varphi_t}) \bigr) \times \{0\}.
\]
\end{lemma}
\begin{proof}
 By Lemma \ref{lem NS}, an element of $\cN^*(S_{x^{-1}})$ is of the form 
 \[
 \bigl( (t,\langle \zeta, u(m)\rangle) , \zeta, -T_m(\varphi_{-t}x)^*\zeta)\bigr),
 \]
for $t \in \R\setminus \{0\}$ and $\zeta \in T_{m}M$. Suppose that this element also lies in $\cN^*(\R\setminus\{0\} \times \Delta(M))$. Then $(m,x\varphi_t(m)) \in \Delta(M)$, so $x\circ \varphi_t(m) = m$. Furthermore,
\[
 \cN^*(\R\setminus\{0\} \times \Delta(M)) = (\Rnz \times \{0\}) \times
 \{(\xi, -\xi); m \in M, \xi \in T^*_mM\},
\]
and it follows that
\[
\begin{split}
 \langle \zeta, u(m)\rangle&= 0 \\
(T_m(x\circ \varphi_{-t}) - 1)^*\zeta &= 0.
\end{split}
\]
We thus see that the $x$-nondegeneracy property of  $\varphi$ implies that $\zeta = 0$.
%
%
%%%%%%
%
%Suppose that $t \in \Rnz$, $a \in T^*_t(\Rnz) = \R$, $m,m' \in M$, $\zeta \in T_m^*M$ and $\xi' \in T_{m'}^*M$. Suppose that
%\[
%( (t,s), \xi, \xi') \in \cN^*(S) \cap \cN^*(\R\setminus\{0\} \times \Delta_g(M)). 
%\]
%Then, first of all, $m' = \varphi_t(m)$ and $g\circ \varphi_t (m) = m$. And by Lemmas \ref{lem NS} and \ref{lem NDelta},
%\[
%\begin{split}
%a &= -\langle \xi, u(\varphi_t(m)\rangle\\
%\end{split}
%\]
\end{proof}

\begin{proof}[Proof of Proposition \ref{prop restr K}]
%For the restriction of $hgh^{-1}\cdot AK$ to $\R\setminus\{0\} \times \Delta(M)$  to be well-defined, it is sufficient that the wave front set $\WF(hgh^{-1}\cdot AK)$ of $hgh^{-1}\cdot K$ is disjoint from the conormal bundle to $\R\setminus\{0\} \times \Delta(M)$. (See Theorem 8.2.4 in \cite{Hormander1}.) And 
Since $hgh^{-1}\cdot AK$ is a $\delta$-distribution on the set $S_{hg^{-1}h^{-1}}$ in \eqref{eq def S}, 
\beq{eq WF K}
\WF(hgh^{-1} \cdot AK) = \cN^*(S_{hg^{-1}h^{-1}}) \setminus \{0\}.
 \eeq
 (See Example 8.2.5 in \cite{Hormander1}.)
 By equivariance of $\varphi_t$, the condition \eqref{eq flow nondeg} implies that $\det\bigl( (T_{m}(hg^{-1}h^{-1}\circ \varphi_t) - 1)_{T_{m}M/\R u(m)} \bigr)\not=0$ for all $h \in G$. So $\varphi$ is $hgh^{-1}$-nondegenerate for all $h \in G$. 
Hence the proposition follows from Lemma \ref{lem NS NDelta}.
\end{proof}

%We will use an equivalent expression for $\Tr^{\flat}_g(A \Phi^*)$ in the proof of Proposition \ref{prop fixed pt Trb}. (We only apply the lemma in the case where $E$ is the trivial line bundle, and with $A$ the identity, but we state it in general here.)
%Recall the notation \eqref{eq def Delta x}.

In the special case where $E$ is the trivial line bundle, with the trivial action by $G$ on fibres, and with $A$ the identity, we will use an equivalent expression for $\Tr^{\flat}_g(\Phi^*)$. Recall the notation \eqref{eq def Delta x}.
\begin{lemma}\label{lem flat trace trivial}
Suppose that $E = M \times \C$, with the trivial $G$-action on fibres. 
Let $K$ be as in \eqref{eq def x K}.  If $\varphi$ is $g$-nondegenerate, then $K$ has a well-defined restriction to $\Rnz \times \Delta_{hg^{-1}h^{-1}}$ for all $h \in G$. Furthermore, 
\beq{eq flat trace trivial}
\langle \Tr^{\flat}_g( \Phi^*), \psi \rangle = 
\int_{G/Z} \langle   K|_{\Rnz \times \Delta_{hg^{-1}h^{-1}}(M)}, \psi \otimes (p_2^*\chi)|_{\Delta_{hg^{-1}h^{-1}}(M)} \rangle\, d(hZ)
\eeq
for all $\psi \in C^{\infty}_c(\Rnz)$.
\end{lemma}
\begin{proof}
Analogously to the proof of Proposition \ref{prop restr K}, one can show that
\[
\WF( K) = \cN^*(S_{e}) \setminus \{0\}
\]
and that for all $x \in G$ 
\[
\cN^*(\Delta_{x}(M)) = \{  (-(T_mx^{-1})^*\xi, \xi); m \in M, \xi \in T^*_mM\}.
\]
This implies that $\WF( K)  \cap \cN^*(\Rnz \times \Delta_{x}(M)) =\emptyset$, as in the proof of Proposition \ref{prop restr K}.

For $\psi \in C^{\infty}_c(\Rnz)$ and $f \in C^{\infty}_c(M)$, and a smooth kernel $K$,  the equality
 \beq{eq flat trace trivial 2}
 \langle x\cdot  K|_{\Rnz \times \Delta(M)} , \psi \otimes (p_2^*f)|_{\Delta(M)}\rangle = 
  \langle   K|_{\Rnz \times \Delta_{x^{-1}}(M)} , \psi \otimes (p_2^*f)|_{\Delta_{x^{-1}}(M)}\rangle
 \eeq
follows from the definitions. Analogously to the proof of Lemma \ref{lem Trg shift}, both sides of this equality depend continuously on $K$ in a space of distributions in which the smooth sections are dense. Hence \eqref{eq flat trace trivial 2} 
extends to distributional kernels by continuity.
\end{proof}

\subsection{A relation with the classical flat trace} \label{sec class flat trace}

The proof of this Proposition \ref{prop KX} will be given after we establish two key lemmas.

Since the covering map $q_M\colon M \to X$ is a submersion, any distributional section can be pulled back along $q_M$. 
\begin{lemma}\label{lem q star u}
For  $v \in \Gamma^{\infty}(E_X^*)'$, and $s \in \Gamma^{\infty}_c(E^*)$,
\beq{eq q star u}
\langle q_M^*v, s\rangle = \Bigl\langle v, \sum_{\gamma \in \Gamma} \gamma \cdot s\Bigr\rangle.
\eeq
\end{lemma}
\begin{proof}
If $v \in \Gamma^{\infty}(E_X)$, then
\begin{align}
\langle q_M^*v, s\rangle &= \int_M \bigl\langle  s(m), v(\Gamma m) \bigl\rangle \, dm \nonumber \\
&= \sum_{\gamma \in \Gamma} \int_M \chi(\gamma m)\bigl\langle  s(m), v(\Gamma m) \bigl\rangle\, dm.\label{eq q star u 2}
\end{align}
Here $\chi$ is as in \eqref{eq cutoff fn}, but in the current setting it may also be replaced by the indicator function of a fundamental domain. In the pairings  in the integrands, we identified $ E_m \cong (E_X)_{\Gamma m} \cong (E/\Gamma)_{\Gamma m}$ via $w \mapsto \Gamma w$, for $w \in E$.

Substituting $m' = \gamma m$, we find that \eqref{eq q star u 2} equals
\[
 \sum_{\gamma \in \Gamma} \int_M \chi( m')
 \bigl\langle \gamma s(\gamma^{-1}m'), v(\Gamma m')\bigr\rangle \, dm' = 
\int_X 
\Bigl\langle  \sum_{\gamma \in \Gamma} (\gamma \cdot s)(m'), v(\Gamma m')\Bigr\rangle
\, d(\Gamma m'), 
\]
which equals the right hand side of \eqref{eq q star u}.

The case for a general distributional section $v$ follows by continuity of both sides of \eqref{eq q star u} in $v$.
\end{proof}

Let $K_X \in \Gamma_c^{\infty}(\Rnz \times X \times X, \Rnz \times E_X^* \boxtimes E_X)'$ be defined as in \eqref{eq def x K} (with $x = e$ and $A$ the identity), for the flow $\Phi_X$. Similarly, let the distributional sections $K$ and $\gamma \cdot K$, for $\gamma \in \Gamma$, in $ \Gamma_c^{\infty}(\Rnz \times M \times M, \Rnz \times E^* \boxtimes E)'$ be as in \eqref{eq def x K}.
%
%%%%%%% Old %%%%%%
%
%For $t \in \Rnz$ we write $K^t$ and $K_X^t$ for the Schwartz kernels of the pullbacks along $\Phi_t$ and $\Phi_{X,t}$, respectively. 
%For $\gamma \in \Gamma$, we define the distribution $\gamma \cdot K^t$ on $M \times M$ by 
%\[
%\langle \gamma\cdot K^t, s\rangle := \langle  K^t, \gamma\cdot s\rangle,
%\]
%where $s \in \Gamma^{\infty}_c(E \boxtimes E)$, and for $m,m' \in M$,
%\[
%(\gamma\cdot s)(m,m') = (1_{E_M} \otimes \gamma)s(m, \gamma^{-1}m').
%\]
%This sum has finitely many nonzero terms, by properness of the action and compactness of the support of $s$. 
The main computation in the proof of Proposition \ref{prop KX} is the following.
\begin{lemma}\label{lem KX KM}
We have
\[
q_{\Rnz \times M \times M}^* K_X = \sum_{\gamma \in \Gamma} \gamma \cdot K.
\]
In particular, the right hand side converges 
in $ \Gamma_c^{\infty}\bigl(\Rnz \times M \times M, (\Rnz) \times E^* \boxtimes E\bigr )'$.
\end{lemma}
\begin{proof}
Let $\psi \in C^{\infty}_c(\Rnz)$, $\xi \in \Gamma^{\infty}_c(E^*)$ and $s \in \Gamma^{\infty}_c(E)$. Then
 by Lemma \ref{lem q star u} (with $M$ replaced by $\Rnz \times M \times M$) and unpacking definitions, we find that
\beq{eq KX KM 1}
\begin{split}
\langle q_{\Rnz \times M \times M}^* K_X, \psi \otimes \xi \otimes s \rangle &=
\left\langle K_X, \sum_{\gamma, \gamma' \in \Gamma} \psi \otimes  \gamma \cdot \xi \otimes \gamma' \cdot s \right\rangle \\
&=  \int_X \int_{\R}\psi(t)
\sum_{\gamma, \gamma' \in \Gamma}
\bigl\langle \gamma \xi(\gamma^{-1}m),  \Phi_{-t}\gamma's(\gamma'^{-1}\varphi_t(m))  \bigr\rangle
dt\, d(\Gamma m)\\
&= 
\sum_{\gamma, \gamma' \in \Gamma}
 \int_M 
 \chi(m) \int_{\R}\psi(t)
\bigl\langle \gamma \xi(\gamma^{-1}m),  \Phi_{-t}\gamma's(\gamma'^{-1}\varphi_t(m))  \bigr\rangle
dt\, dm,
%%%%%%%%%%%%%%%%%
%&= \int_X \left(s_1^X(\Gamma m), \Phi^X_{-t} s_2^X(\varphi_t^X(\Gamma m)) \right)\, d(\Gamma m)\\
%&= \int_X \sum_{\gamma, \gamma' \in \Gamma} \bigl(\gamma s_1(\gamma^{-1}m), \gamma' (\Phi_t^*s_2) (\gamma'^{-1}m)\bigr)\, d(\Gamma m)\\
%&= \int_M \chi(m)  \sum_{\gamma, \gamma' \in \Gamma} \bigl(\gamma s_1(\gamma^{-1}m), \gamma' (\Phi_t^*s_2) (\gamma'^{-1}m)\bigr)\, dm.
\end{split}
\eeq
for a 
 a cutoff function $\chi$.
%  close enough to the
% indicator function of a fundamental domain, we can make the latter expression arbitrarily close to
%\beq{eq KX KM 1}
%\sum_{\gamma, \gamma' \in \Gamma}
% \int_M 
% \chi(m) \int_{\R}\psi(t)
%\bigl\langle \gamma \xi(\gamma^{-1}m),  \Phi_{-t}\gamma's(\gamma'^{-1}\varphi_t(m))  \bigr\rangle
%dt\, dm.
%\eeq
Substituting $m' = \gamma^{-1}m$, $\gamma'' = \gamma^{-1}\gamma'$, and using $\Gamma$-invariance of the metric on $E$ and $\Gamma$-equivariance of $\Phi_t$ and $\varphi_t$, we find that the right hand side of  \eqref{eq KX KM 1} equals
\[
\begin{split}
\sum_{\gamma, \gamma'' \in \Gamma}
 \int_M   \chi(\gamma m')  %\int_{\Rnz} \psi(t)
\langle  \xi(m'),  (\gamma' \Phi_{\psi}^* s)(m')\rangle &= 
\sum_{\gamma'' \in \Gamma}
 \int_M     %\int_{\Rnz} \psi(t)
\langle  \xi(m'),  (\gamma' \Phi_{\psi}^* s)(m')\rangle
dm'\\
&= \sum_{\gamma'' \in \Gamma} \langle \gamma' \cdot K, \psi \otimes \xi \otimes s\rangle. 
%
%
% \int_M  \sum_{\gamma, \gamma'' \in \Gamma} \chi(\gamma m')  \bigl( s_1(m'), \gamma'' (\Phi_t^*s_2) (\gamma''^{-1}m')\bigr)\, dm'
% &=
%  \int_M  \sum_{\gamma'' \in \Gamma} \bigl( s_1(m'), (\Phi_t^*(\gamma''\cdot s_2) (m')\bigr)\, dm'\\
%  &= \sum_{\gamma'' \in \Gamma} \langle K, s_1 \boxtimes (\gamma'' \cdot s_2)\rangle\\
%  &= \left\langle \sum_{\gamma'' \in \Gamma} \gamma'' \cdot K, s_1 \boxtimes s_2 \right\rangle.
\end{split}
\]
\end{proof}

\begin{proof}[Proof of Proposition \ref{prop KX}.]
Commutativity of the diagram 
\[
\xymatrix{
\Delta(M) \ar@{^{(}->}[r] \ar[d]^-{q_{\Delta(M)}} & M \times M  \ar[d]^-{q_{M \times M}}  \\ 
\Delta(X) \ar@{^{(}->}[r]& X \times X
}
\]
and Lemma \ref{lem KX KM}
imply that
\begin{align}
q_{\Rnz \times \Delta(M)}^*(K_X|_{\Rnz \times \Delta(X)}) &= (q_{\Rnz \times M \times M}^*K_X)|_{\Rnz \times\Delta(M)}\nonumber \\
&=  \sum_{\gamma \in \Gamma} (\gamma \cdot K)|_{\Rnz \times\Delta(M)}
\label{eq q star K X}.
\end{align}

%\beq{eq q star K X}
%q_{\Rnz \times \Delta(M)}^*(K_X|_{\Rnz \times \Delta(X)}) = (q_{\Rnz \times M %
%\times M}^*K_X)|_{\Rnz \times\Delta(M)} = \sum_{\gamma \in \Gamma} (\gamma \cdot 
%K)|_{\Rnz \times\Delta(M)}.
%\eeq
%By Lemma \ref{lem KX KM}, the right hand side equals
%\[
%\sum_{\gamma \in \Gamma} (\gamma \cdot K)|_{\Rnz \times\Delta(M)}.
%\]
Let  $\psi \in C^{\infty}_c(\Rnz)$. We view the $\Gamma$-invariant function $ \sum_{\gamma \in \Gamma} \gamma \cdot (p_2^*\chi)|_{\Delta(M)}$ on $\Delta(M)$ (which is constant $1$) as a function on $\Delta(X)$. Using this identification, and then Lemma \ref{lem q star u} and \eqref{eq q star K X}, we find that
\[
\begin{split}
\langle \Tr^{\flat}(\Phi_X), \psi \rangle  &=
\langle K_X|_{\Rnz \times \Delta(X)}, \psi \otimes 1_{\Delta(X)}\rangle \\
&= 
\Bigl\langle K_X|_{\Rnz \times \Delta(X)}, \psi \otimes \sum_{\gamma \in \Gamma} (\gamma \cdot (p_2^*\chi)|_{\Delta(M)})\Bigr\rangle \\
&= \Bigl\langle q^*(K_X|_{\Rnz \times \Delta(X)}), \psi \otimes  (p_2^*\chi)|_{\Delta(M)}\Bigr\rangle \\
&=
\sum_{\gamma \in \Gamma} 
\Bigl\langle \gamma \cdot K|_{\Rnz \times\Delta(M)}, \psi \otimes  (p_2^*\chi)|_{\Delta(M)}\Bigr\rangle\\
&= \sum_{(\gamma)}\langle \Tr^{\flat}_{\gamma}(\Phi), \psi \rangle.
\end{split}
\]
\end{proof}

%
% We will use the expression for $\Tr_g^{\flat}(\Phi^*)$ in Lemma \ref{lem flat trace trivial}. 

\section{Proof of the equivariant trace formula}\label{sec fixed pt}

%Guillemin's trace formula \cite{Guillemin77} is an expression for the (non-equivariant)
%flat trace of the pullback map of functions along a flow, in terms  of the closed orbits of this flow. 
%Our goal in this section is to prove an equivariant generalisation of this formula, 
%Theorem \ref{prop fixed pt gen}. This involves the flat $g$-trace from Section \ref{sec flat g trace}. 

In this section, we prove the equivariant version of Guillemin's trace formula, Theorem \ref{prop fixed pt gen}.

\subsection{A local expression near individual fixed points}

Until stated otherwise in Subsection \ref{sec pf fixed pt}, 
we assume that $E = M \times \C$ with the trivial action on fibres. We take $A$ to be the identity.
Let the Schwartz kernel $K$ be as in \eqref{eq def x K}, for $x=e$. 
From now on, we will often use notation as if $K$ were a function.
%
%, for example writing
%\[
%\int_M \int_\R K(t, xm, m) \psi(t) \chi(m)\, dt \, dm := \langle K|_{\Rnz \times \Delta_x(M)}, \psi \otimes (p_2^*\chi)|_{\Delta_x(M)} \rangle.
%\]

We fix a $G$-invariant Riemannian metric on $M$, and use it to identify $T_mM/\R u(m) \cong u(m)^{\perp}$ for $m \in M$. When $x \in G$ and $l \in \Rnz$ are given, we also assume that
this Riemannian metric is preserved by $\varphi$ along flow curves in $\Gamma_l^x(\varphi)$.  
This is possible by Lemma \ref{lem Gammagl discrete} and the use of a suitable partition of unity. 
 Then for all such flow curves $\gamma$,  the map $T_{\gamma(0)} (\varphi_l \circ x^{-1})$ preserves 
 $u(\gamma(0))^{\perp} \subset T_{\gamma(0)}M$, and 
 $P^x_{\gamma}$ can be identified with 
\[
T_{\gamma(0)} (\varphi_l \circ x^{-1})|_{u(\gamma(0))^{\perp}}.
\]

We start with a generalisation of the local computation carried out in Lemma B.1 in \cite{DZ16}. For $p \in \R^n$ and $r>0$, we denote the open Euclidean ball in $\R^n$ with centre $p$ and radius $r$ by $B^n(p, r)$.
\begin{lemma}\label{lem fixed pt loc}
Let $m_0 \in M$, $x \in G$ and $l \in \Rnz$ be such that $\varphi_{l}(m_0) = xm_0$. 
 Suppose that $\varphi$ is $x$-nondegenerate. Then   
there exists a density $dm$ on $M$, an  $\varepsilon>0$, and a neighbourhood $U$ of $m_0$ such that $\varphi_s(m_0) \in U$ for all 
$s \in (-\varepsilon, \varepsilon)$, such that  for all 
$f \in C_c^{\infty}( (l-\varepsilon,  l+\varepsilon) \times U)$, 
\begin{multline}\label{eq fixed pt lem}
%\Tr^{\flat}_{g}(K), f\rangle 
\int_{\R \times M} f(t,m) K(t, x^{-1}m,m)\, dt\, dm
= \\
\frac{1}{|\det(1-T_{m_0} (\varphi_{l} \circ x^{-1})|_{u(m_0)^{\perp}})   |}
\int_{-\varepsilon}^{\varepsilon} f(l, \varphi_s(m_0))\, ds.
\end{multline}
\end{lemma}

%\Todo: where used, comment that relevant expression indep of density.
\begin{proof}
%The proof below is similar to the proof of Lemma B.1 in \cite{DZ16}.
%
Choose a neighbourhood $V$ of $m_0$, an $r>0$, and a chart $\kappa \colon V \to B^n(0, r)$ such that
\[
\begin{split}
\kappa(m_0)&= 0;\\
\kappa_* u &= \frac{\partial}{\partial x_1};\\
T_{m_0}\kappa(u(m_0)^{\perp}) &= \{0\} \times \R^{n-1}.
\end{split}
\]
Fix $\varepsilon \in (0,r)$ such that the neighbourhood $U:= \kappa^{-1}(B^n(0,\varepsilon))$ has the property that for all $t \in (l - 2\varepsilon, l + 2\varepsilon)$,
\[
x^{-1} \circ \varphi_{l} (U) \subset V.
\]
Define the map
\[
F\colon B^{n-1}(0,\varepsilon) \to B^n(0,r)
\]
by commutativity of the diagram
\[
\xymatrix{
\kappa^{-1}(\{0\}\times B^{n-1}(0,\varepsilon)) \ar[r]^-{\varphi_{l} \circ x^{-1}} \ar[d]_-{\kappa} & V \ar[d]_-{\kappa}\\
\{0\}\times B^{n-1}(0,\varepsilon) \ar[r]^-{F} & B^n(0,r).
}
\]
Let $F_1\colon B^n(0,\varepsilon) \to(-r,r)$ be the first component of $F$, and $F'\colon B^n(0,\varepsilon) \to B^{n-1}(0,r)$ the last $n-1$ components.

Let $t \in (l - \varepsilon, l+\varepsilon)$. For all $x_1 \in \R$ and $x' \in \R^{n-1}$ such that $(x_1, x') \in B^n(0,\varepsilon)$, the facts that the flow $\varphi$ commutes with $x$ and that $|x_1|$ and $|l-t+x_1|$ are smaller than $2\varepsilon$ imply that
\[
\begin{split}
(\varphi_{t} \circ x^{-1})\kappa^{-1}(x_1, x') &= \varphi_{t-l} \circ (\varphi_{l} \circ x^{-1}) \circ \varphi_{x_1}\circ  \kappa^{-1}(0,x')\\
&= \varphi_{t-l+x_1}  \kappa^{-1}(F(x'))\\
&= \kappa^{-1}(t-l+x_1 + F_1(x'), F'(x')).
\end{split}
\]
Choose a density $dm$ on $M$ such that $\kappa_* (dm)$ is the Lebesgue measure on $V$. Then the above computation implies that 
the Schwartz kernel $K^{x}_{t}$ of the pullback along $\varphi_{t} \circ x^{-1}$ is locally given by
\[
K^{x}_{t}(\kappa^{-1}(x_1, x'), \kappa^{-1}(y_1, y')) = 
\delta_{\R^{n-1}}(y'-F'(x'))\delta_{\R}(y_1 + l-t-x_1 - F_1(x')).
\]
Note that 
\[
K(t,x^{-1}m,m') = K_{t}^{x}(m,m'). 
\]
So for $f \in C_c( (l-\varepsilon,  l+\varepsilon) \times U)$,
\begin{multline}\label{eq fixed pt lem 1}
\int_{\R \times M} f(t,m) K(t, x^{-1}m,m)\, dt\, dm
= \\
\int_{\R \times B^n(0,\varepsilon)} f(t,\kappa^{-1}(x_1, x')) 
\delta_{\R^{n-1}}(x'-F'(x'))\delta_{\R}( l-t - F_1(x'))
\, dt\, dx_1\, dx'=\\
\int_{B^n(0,\varepsilon)} f(l - F_1(x'),\kappa^{-1}(x_1, x')) 
((1-F')^* \delta_{\R^{n-1}})(x')
\, dx_1\, dx'.
\end{multline}

We have a commuting diagram
\[
\xymatrix{
u(m_0)^{\perp} \ar[rr]^-{T_{m_0}(\varphi_{l} \circ x^{-1})} \ar[d]_-{T_{m_0}\kappa} & & u(m_0)^{\perp}  \ar[d]_-{T_{m_0}\kappa}\\
\R^{n-1} \ar[rr]^-{T_0F'} && \R^{n-1}.
}
\]
So
\[
\det(1-T_{m_0}(\varphi_{l} \circ x^{-1})|_{u(m_0)^{\perp}})= \det(1-T_0 F'),
\]
and in particular the right hand side is nonzero. This implies that $\det(1-T_{x'} F')\not=0$ for all $x' \in B^{n-1}(0, \varepsilon)$, if
 we choose $\varepsilon$ small enough. Then, again for small enough $\varepsilon$, the only point $x' \in B^{n-1}(0, \varepsilon)$ such that $(1-F')(x') = 0$ is the origin. For such an $\varepsilon$, we find that the right hand side of \eqref{eq fixed pt lem 1} equals
 \[
\frac{1}{|\det(1-T_{m_0} (\varphi_{l} \circ x^{-1})|_{u(m_0)^{\perp}})  |}
\int_{-\varepsilon}^{\varepsilon}
 f(l,\kappa^{-1}(x_1, 0)) 
\, dx_1, 
 \]
 which equals the right hand side of \eqref{eq fixed pt lem}.
\end{proof}

\subsection{Globalising Lemma \ref{lem fixed pt loc}}

In this subsection we will establish a 
global version of Lemma \ref{lem fixed pt loc}. 
Recall that we write $I_{\gamma}$ for an interval on which a curve $\gamma\colon \R \to M$ is bijective onto its image, up to sets of measure zero.
\begin{lemma}\label{lem fixed pt glob}
Let $x \in G$. Suppose that $\varphi$ is $x$-nondegenerate.  Suppose that $L_{x}(\varphi)$ is discrete. 
Then for all $f \in C^{\infty}_c(\Rnz \times M)$,
\begin{multline}\label{eq fixed pt lem glob}
\int_{\R \times M} f(t,m) K(t, x^{-1}m,m)\, dt\, dm
= \\
\sum_{l \in L_{x}(\varphi)} \sum_{\gamma \in \Gamma^{x}_l(\varphi)}
\frac{1}{|\det(1-P_{\gamma}^{x})  |}
\int_{I_{\gamma}} f(l, \gamma(s))\, ds.
\end{multline}
\end{lemma}
\begin{proof}
%The proof is a reduction to Lemma \ref{lem fixed pt loc} via partitions of unity and a fair amount of book-keeping. 
%
Let $\Delta_{x^{-1}}(m) = (x^{-1}m,m)$, for $m \in M$. 
The support of $K|_{\Rnz \times \Delta_{x^{-1}}(M)}$ is 
\[
\{(t,x^{-1}m,m); m \in M, \varphi_t(m) = xm\} = 
 \bigcup_{l \in L_x(\varphi)} \bigcup_{\gamma \in \Gamma_l^x(\varphi)} \{l\} \times \Delta_{x^{-1}}(\im(\gamma)).
% = 
%\coprod_{l \in L_{g}(\varphi)} \{l\} \times \Delta_{g^{-1}}(M^{g^{-1}\varphi_l}),
\]
%where $M^{g^{-1}\varphi_l}$ is the fixed-point set of $g^{-1}\varphi_l$, and 
%\[
%\Delta_{g^{-1}}(M^{g^{-1}\varphi_l}) = \{(g^{-1}m,m); m \in M^{g^{-1}\varphi_l}\}.
%\]
%So if $p_3 \colon \Rnz \times M \times M \to M$ is projection onto the third factor, then 
%\begin{multline} \label{eq p3 supp K}
%p_3(\supp(K|_{\Rnz \times \Delta_{g^{-1}}})) =
%\{m \in M; \varphi_l(m) = gm \text{ for some $l \not=0$}\} \\
%= 
% \bigcup_{l \in L_g(\varphi)} \bigcup_{\gamma \in \Gamma_l^g(\varphi)} \im(\gamma).
%\end{multline}
The unions are disjoint. 
The set $L_x(\varphi)$ is discrete by assumption, and for each $l \in L_x(\varphi)$, the set $\Gamma_l^x(\varphi)$ is discrete by Lemma \ref{lem Gammagl discrete}. Via a partition of unity, the left hand side of \eqref{eq fixed pt lem glob} can be written as a sum of contributions from $l \in L_x(\varphi)$ and $\gamma \in \Gamma_l^x(\varphi)$. We will use Lemma \ref{lem fixed pt loc}  to show that such a contribution equals the corresponding term on the right hand side of \eqref{eq fixed pt lem glob}.

Fix $l \in L_{x}(\varphi)$ and $\gamma \in \Gamma_l^x(\varphi)$.
The set $\gamma([0,l])$ is compact, so there are finitely many $t_1, \ldots, t_k \in (0,l)$, $\varepsilon_1, \ldots, \varepsilon_k>0$ and neighbourhoods $U_1, \ldots, U_k$ of $\gamma(t_1), \ldots, \gamma(t_k)$, respectively, 
as in Lemma \ref{lem fixed pt loc} applied with $m_0 = \gamma(t_1), \ldots, \gamma(t_k)$, such that  
\[
\gamma([0,l]) \subset \bigcup_{j=1}^k U_j.
\]
In fact, to simplify notation, we choose all the numbers $\varepsilon_j$ to be equal to the same number $\varepsilon$. We also choose the number $\varepsilon$ small enough so that $(l-\varepsilon, l+\varepsilon) \cap L_{x}(\varphi) = \{l\}$. %(Here we use discreteness of $L_{x}(\varphi)$.)

Since $\Gamma_l^x(\varphi)$ is discrete, we can
choose the sets $U_j$ so that 
\[
\bigcup_{j=1}^k U_j \cap \bigcup_{\tilde \gamma \in \Gamma_l^x(\varphi)} \im(\tilde \gamma) \subset \im(\gamma),
\]
i.e.\ $\bigcup_{j=1}^k U_j$ does not intersect any curves in $\Gamma_l^x(\varphi)$ apart from $\gamma$.
For $n \in \Z$ and $j=1,\ldots, k$, 
set
\[
\begin{split}
t_{n,j}&:= t_j + nl \\
%\varepsilon_{n,j} &:= \varepsilon_j;\\
U_{n,j} &:= x^{-n}U_j.
\end{split}
\]
The fact that $\gamma([nl, (n+1)l]) = x^{-n}\gamma([0,l])$ then implies that
\beq{eq union Unj}
\bigcup_{n \in \Z}\bigcup_{j=1}^k U_{n,j} \cap \bigcup_{\tilde \gamma \in \Gamma_l^x(\varphi)} \im(\tilde \gamma) 
= \im(\gamma).
\eeq
And the number $\varepsilon >0$ and the neighbourhoods $U_{n,j}$ of $\gamma(t_{n,j})$ are as in Lemma \ref{lem fixed pt loc}.

We choose a sub-collection of the sets $U_{n,j}$ that cover $\im(\gamma)$ in a locally finite way. The interval $I_{\gamma}$ is either all of $\R$ (if $\gamma$ is  not periodic) or bounded (if $\gamma$ is periodic). If $I_{\gamma} = \R$, then we use all sets $U_{n,j}$, and we set
\[
J:= \Z \times \{1,\ldots, k\}.
\]
If $I_{\gamma}$ is bounded, we can choose it to be of the form $[0,T]$ for some $T>0$. Then let $n_1, n_2 \in \Z$ be nonnegative such that $n_2 \leq k$,  $t_{n_1, n_2} \geq T$ and $t_{n_1, n_2-1} < T$ (if $n_2\geq 1$) or $t_{n_1-1, k} < T$ (if $n_2=1$). Then
\[
I_{\gamma} \subset [t_{-1,k}, t_{n_1, n_2}].
\]
We now set
\[
J := \{ (-1,k) \}  \cup \bigl(\{0,\ldots, n_1-1\} \times \{1,\ldots, k\} \bigr) \cup \bigl( \{n_1\} \times \{1,\ldots, n_2\} \bigr).
\]
The middle component is omitted if $n_1 = 0$, which happens if $T<l$. Then the set $\{t_\alpha; \alpha \in J\}$ consists of consecutive elements of the set of all $t_{n,j}$, and
\[
\begin{split}
I_{\gamma} \subset \bigl[\min_{\alpha \in J} t_{\alpha},   \max_{\alpha \in J} t_{\alpha} \bigr];
\\
\bigcup_{\alpha \in J} U_{\alpha} \cap \bigcup_{\tilde \gamma \in \Gamma_l^x(\varphi)} \im(\tilde \gamma) 
&= \im(\gamma).
\end{split}
\]

% Because there are only finitely many different $\varepsilon_{n,j}$, we have
%\[
%\varepsilon := \inf\{\varepsilon_{n,j}\}>0.
%\]

Let $\zeta_j \in C^{\infty}_c(U_j)$ and set $\zeta_{n,j}:= (x^{-n})^{*}\zeta_j$. By construction of the set $J$ in both the cases $I_{\gamma} = \R$ and $I_{\gamma}$ bounded, we can choose these functions so that
\[
\zeta := \sum_{\alpha \in J} \zeta_{n,j}
\]
equals $1$ on the image of $\gamma$, and zero on the images of all other curves in $\Gamma_l^x(\varphi)$.
%
%By \eqref{eq union Unj}, we can choose these functions so that their sum, which we denote by $\zeta$, equals $1$ on the image of $\gamma$, and zero on the rest of $\supp(K|_{\Rnz \times \Delta_{g^{-1}}})$.
%
Let $\psi \in C_c^{\infty}(l-\varepsilon, l+\varepsilon)$ be such that $\psi(l)=1$. Then Lemma \ref{lem fixed pt loc} implies that for all $f \in C^{\infty}_c(\Rnz \times M)$,
\begin{multline}\label{eq fixed pt contr gamma}
\int_{\R \times M} f(t,m) \psi(t)\zeta(m) K(t, x^{-1}m,m)\, dt\, dm
= \\
%\sum_{n \in \Z} \sum_{j=1}^{k}
\sum_{\alpha \in J}
\frac{1}{|\det(1-T_{\gamma(t_{\alpha})} (\varphi_{l} \circ x^{-1})|_{u(\gamma(t_{\alpha}))^{\perp}})  |}
\int_{-\varepsilon}^{\varepsilon} f(l, \varphi_s(\gamma(t_{\alpha})))   \zeta_{\alpha}(\varphi_s(\gamma(t_{\alpha})))\, ds.
\end{multline}

Now for all $n$ and $j$
\[
T_{\gamma(t_{\alpha})} (\varphi_{l} \circ x^{-1})|_{u(\gamma(t_{\alpha}))^{\perp}} = 
T_{\gamma(0)} \varphi_{t_{\alpha}} \circ
P^{x}_{\gamma} \circ
T_{\gamma(t_{\alpha})} \varphi_{-t_{\alpha}}, 
\]
so
\[
\det(1-T_{\gamma(t_{\alpha})} (\varphi_{l} \circ x^{-1})|_{u(\gamma(t_{\alpha}))^{\perp}}) = 
\det(1-P^{x}_{\gamma}).
\]

For any interval $I\subset \R$, we write $1_I$ for the function on $\R$ that is $1$ on $I$ and $0$ outside $I$.
The expression
\[
\sum_{\alpha \in J} 
\int_{-\varepsilon}^{\varepsilon} f(l, \varphi_s(\gamma(t_{\alpha})))   \zeta_{\alpha}(\varphi_s(\gamma(t_{\alpha})))\, ds 
%= 
%\sum_{\alpha \in J} 
%\int_{-\varepsilon}^{\varepsilon} f(l, \gamma(t_{\alpha}+s))  \psi(l) \zeta_{\alpha}(\gamma(t_{\alpha}+s))\, ds 
=
\int_{\R}  f(l, \gamma(s))   \sum_{\alpha \in J}  1_{[t_{\alpha} - \varepsilon, t_{\alpha} + \varepsilon]} (s) \zeta_{\alpha}(\gamma(s))\, ds 
\]
approaches the number $\int_{I_{\gamma}}f(l, \gamma(s))  \, ds$ arbitrarily closely if the smooth functions $\zeta_j$ are chosen so that the function 
\[
s\mapsto \sum_{\alpha \in J} 1_{[t_{\alpha} - \varepsilon, t_{\alpha} + \varepsilon]}(s) \zeta_{\alpha}(\gamma(s))
\]
approaches $1_{I_{\gamma}}$ closely enough in $L^1$-norm. The left hand side of \eqref{eq fixed pt contr gamma} is independent of the functions $\zeta$ and $\psi$ with the properties mentioned, and therefore equals
%\beq{eq contr gamma fixed pt}
\[
\frac{1 }{|\det(1-P^{x}_{\gamma})|} \int_{I_{\gamma}}f(l, \gamma(s))  \, ds,
\]
which completes the proof.
\end{proof}

\subsection{Proof of Theorem \ref{prop fixed pt gen}}\label{sec pf fixed pt}

We will start by proving Theorem \ref{prop fixed pt gen} in the case where $E$ is the trivial line bundle.
\begin{proposition}\label{prop fixed pt Trb}
Suppose that $E$ is the trivial line bundle. 
Suppose that $\varphi$ is $g$-nondegenerate. 
Then the operator $\Phi^*_{\psi}$ is flat $g$-trace class for all $\psi \in C^{\infty}_c(\Rnz)$.
Furthermore,
\beq{eq fixed pt Trb}
\Tr^{\flat}_g(\Phi^*) = \int_{G/Z}  \sum_{l \in L_{g}(\varphi)} \sum_{\gamma \in \Gamma^{hgh^{-1}}_l(\varphi)} 
\frac{\delta_{l}}{|\det(1-P_{\gamma}^{hgh^{-1}})  |}T_{\gamma} \, d(hZ).
\eeq
%Here $\delta_l \in \calD'(\Rnz)$ is the Dirac $\delta$ at $l$.
\end{proposition}
\begin{proof}
By Lemmas \ref{lem flat trace trivial} and \ref{lem fixed pt glob}, we have for all $\psi \in C^{\infty}_c(\Rnz)$, 
\[
\langle
\Tr^{\flat}_g(\Phi^*), \psi\rangle = \int_{G/Z}  \sum_{l \in L_{g}(\varphi)} \sum_{\gamma \in \Gamma^{hgh^{-1}}_l(\varphi)} 
\frac{1}{|\det(1-P_{\gamma}^{hgh^{-1}})  |} \psi(l) \int_{I_{\gamma}} \chi(\gamma(s))\, ds\, d(hZ).
\]
Here we have also used \eqref{eq Lhghinv}. 
%Using $h\gamma(s) = \gamma(s+l)$ for $\Gamma^{h}_l(\varphi)$, we find that the right hand side equals the right hand side of \eqref{eq fixed pt Trb}.
\end{proof}

\begin{remark}\label{rem G=e}
Suppose that $G = \{e\}$, so $\chi \equiv 1$. 
Then by Example \ref{ex prim per e},  the equality \eqref{eq fixed pt Trb} in Proposition \ref{prop fixed pt Trb} becomes
%does not apply in its current form, because $\varphi$ should have closed flow curves on a compact manifold. But in a version allowing closed curves, the integral over $\R$ should be replaced by an interval $I_{\gamma}$ on which $\gamma$ is a bijection onto its image (modulo sets of measure zero). That version of the proposition says that
\[
\Tr^{\flat}_e(\Phi^*) = \sum_{l \in L_{e}(\varphi)} \sum_{\gamma \in \Gamma^{e}_l(\varphi)} 
\frac{\delta_{l}}{|\det(1-P_{\gamma}^{e})  |} T_{\gamma}^{\#}.
\]
This is Guillemin's trace formula %(1.5) in \cite{DZ16}, which is the same as 
(II.17) in \cite{Guillemin77}; see also (1.5) in \cite{DZ16}.
\end{remark}

Proposition \ref{prop fixed pt Trb} is the special case of Theorem \ref{prop fixed pt gen} when $E$ is the trivial line bundle. We now show how this special case implies 
Theorem \ref{prop fixed pt gen} for general vector bundles $E$ (compare also the end of Appendix B in \cite{DZ16}). Fix $x \in G$. 
Let $A \in \End(E)$, and suppose $A$ commutes with both the $G$-action and with $\Phi$. 
Consider the function $ \tr(x\cdot A\Phi) \in C^{\infty}(\Rnz \times \Delta(M))$, given in terms of a local frame $\{e_j\}_{j=1}^{\rank(E)}$ for $E$, with dual frame $\{e^j\}_{j=1}^{\rank(E)}$, by
\beq{eq tr x Phi}
 \tr(x\cdot A\Phi)(t,m,m) := \sum_{j=1}^{\rank(E)} \left\langle e^j(m),( (x\cdot A\Phi_{-t}) e_j)(m) \right\rangle
\eeq
for $t \in \Rnz$ and $m \in M$.
Let $x\cdot K^0$ be the Schwartz kernel of  the operator given by pulling back scalar functions along $x^{-1} \circ \varphi_t$,  
as in \eqref{eq def x K} with $E = M \times \C$.
\begin{lemma}\label{lem red scalar}
Suppose that $\varphi$ is $g$-nondegenerate. 
We have the following equality in $\calD'(\Rnz \times \Delta(M))$:
\beq{eq red scalar}
\tr\left( x\cdot A K|_{\Rnz \times \Delta(M)} \right) = \tr(x\cdot A\Phi)\left( x\cdot K^0|_{\Rnz \times \Delta(M)} \right).
\eeq
\end{lemma} 
\begin{proof}
Consider a local frame $\{e_j\}_{j=1}^{\rank(E)}$ for $E$ defined in an open set $U \subset M$, with dual frame  $\{e^j\}_{j=1}^{\rank(E)}$ for $E^*$. Extending \eqref{eq tr x Phi}, we define $\tilde \tr(x\cdot A\Phi) \in C^{\infty}(\Rnz \times U\times U)$  by
\[
 \tilde \tr(x\cdot A\Phi)(t,m,m') := \sum_{j=1}^{\rank(E)} \left\langle e^j(m),( (x\cdot A\Phi_{-t}) e_j)(m) \right\rangle,
\]
for $t \in \Rnz$ and $m,m' \in U$. Observe that this function is constant in the second factor $U$.

Define  $\tilde \tr(x\cdot AK) \in \calD'(\Rnz \times U \times U)$ by
\[
\left\langle \tilde \tr(x\cdot AK), \psi \otimes f_1 \otimes f_2\right\rangle = 
\sum_{j}\left\langle x\cdot AK, \psi \otimes f_1e^j \otimes f_2 e_j\right\rangle,
\]
for  $\psi \in C^{\infty}_c(\Rnz)$, and  $f_1, f_2 \in C^{\infty}_c(U)$. Then for all such $\psi$, $f_1$ and $ f_2$, \eqref{eq def Phi psi} and \eqref{eq def x K} imply that
\begin{multline*}
%\begin{split}
\left\langle \tilde \tr(x\cdot AK), \psi \otimes f_1 \otimes f_2\right\rangle\\
 = 
\sum_{j} \int_M \int_{\R} \psi(t) f_1(m) f_2(\varphi_t x^{-1}m) \left\langle e^j(m), xA\Phi_{-t}e_j(\varphi_t( x^{-1}m))\right\rangle\, dt\, dm\\
%
%
%
%
%
%
%&= \sum_{j=1}^{\rank(E)} \left\langle   x\cdot K|_{\Rnz \times \Delta(M)}  , \psi \otimes f e_j \otimes e^j \right\rangle\\
%%&= \int_M \int_{\R} \psi(t) f(\varphi_t (x^{-1} m)) \left\langle e^j(m), x\Phi_{-t} e_j(\varphi_t (x^{-1} m))
%%\right\rangle \, dt\, dm\\
%&= \int_M \int_{\R} \psi(t) f(\varphi_t (x^{-1} m)) \tr(x\cdot \Phi)(t,m,m)   \, dt\, dm\\
= \left\langle  x\cdot K^0,  \tilde \tr(x\cdot A\Phi)  \cdot ( \psi \otimes f_1 \otimes f_2) \right\rangle
\end{multline*}
so
\[
 \tilde \tr(x\cdot AK) =\tilde \tr(x\cdot A\Phi) \cdot ( x\cdot K^0).
\]
The restrictions of both sides to $\Rnz \times \Delta(U)$ are well-defined by Proposition \ref{prop restr K}. Taking these restrictions and using
\[
\begin{split}
\tilde \tr(x\cdot A\Phi)|_{\Rnz \times \Delta(U)} &= \tr(x\cdot A\Phi)|_{\Rnz \times \Delta(U)} ;\\
\tilde \tr(x\cdot AK)|_{\Rnz \times \Delta(U)} &=  \tr(x\cdot AK|_{\Rnz \times \Delta(U)}),
\end{split}
\]
and partitions of unity with respect to sets like $U$, 
we obtain \eqref{eq red scalar}.
\end{proof}

Theorem \ref{prop fixed pt gen} follows from Proposition \ref{prop fixed pt Trb}, Lemma \ref{lem red scalar},  and the fact that for all $m \in M$ such that $\varphi_t(m) = xm$,
\[
 \tr(x\cdot A\Phi)(t,m,m) = \tr(Ax(\Phi_{-t})_m), 
\]
the trace of the linear endomorphism $A\circ x\circ (\Phi_{-t})_m$ of $E_m$.
% (\Todo: need explanation of the effect of multiplying by a function?)

 \bibliographystyle{plain}

\bibliography{mybib}

\end{document}